\documentclass[12pt]{amsart}
\usepackage{amscd,amsmath,amssymb,amsfonts,verbatim}
\usepackage{graphicx}
\usepackage[active]{srcltx}

\usepackage{amscd,amsmath,amssymb,amsfonts} 
\usepackage[cmtip, all]{xy}

\newtheorem{thm}{Theorem}[section]
\newtheorem{prop}[thm]{Proposition}
\newtheorem{lem}[thm]{Lemma}
\newtheorem{cor}[thm]{Corollary}

\theoremstyle{definition}
\newtheorem{exm}[thm]{Example}
\newtheorem{defn}[thm]{Definition}

\theoremstyle{remark}
\newtheorem{remk}[thm]{Remark}
\newtheorem{remks}[thm]{Remarks}
\newtheorem{exms}[thm]{Examples}
\newtheorem{notat}[thm]{Notation}

\numberwithin{equation}{section}

{\hfill$\square$\end{defn}}

{\hfill$\square$\end{remk}}

{\hfill$\square$\end{remks}}

{\hfill$\square$\end{exm}}

{\hfill$\square$\end{exms}}

{\hfill$\square$\end{notat}}

\newcommand{\thmref}{Theorem~\ref}
\newcommand{\propref}{Proposition~\ref}

\newcommand{\surj}{\twoheadrightarrow}
\newcommand{\inj}{\hookrightarrow}

\newcommand{\codim}{{\rm codim}}

\newcommand{\Hom}{{\rm Hom}}

\newcommand{\Spec}{{\rm Spec \,}}

\newcommand{\Sch}{{\operatorname{\mathbf{Sch}}}}

\newcommand{\<}{\langle}
\renewcommand{\>}{\rangle}

\newcommand{\Spc}{{\mathbf{Spc}}}
\newcommand{\Sm}{{\mathbf{Sm}}}

\newcommand{\ds}{{/\kern-3pt/}}

\newcommand{\Tor}{{\operatorname{Tor}}}

\newcommand{\colim}{\mathop{\text{colim}}}

\newcommand{\ov}{\overline}
\newcommand{\wt}{\widetilde}
\newcommand{\wh}{\widehat}

\newcommand{\tuborg}{\left\{\begin{array}{ll}}
\newcommand{\sluttuborg}{\end{array}\right.}

\newcommand{\sC}{{\mathcal C}}
\newcommand{\sD}{{\mathcal D}}

\newcommand{\sF}{{\mathcal F}}
\newcommand{\sG}{{\mathcal G}}
\newcommand{\sH}{{\mathcal H}}

\newcommand{\sK}{{\mathcal K}}

\newcommand{\sO}{{\mathcal O}}

\newcommand{\sS}{{\mathcal S}}

\newcommand{\sU}{{\mathcal U}}
\newcommand{\sV}{{\mathcal V}}

\newcommand{\A}{{\mathbb A}}

\newcommand{\C}{{\mathbb C}}

\newcommand{\G}{{\mathbb G}}

\newcommand{\N}{{\mathbb N}}
\renewcommand{\P}{{\mathbb P}}

\newcommand{\Z}{{\mathbb Z}}

\begin{document}
\title{The completion problem for equivariant $K$-theory}
\author{Amalendu Krishna}
\address{School of Mathematics, Tata Institute of Fundamental Research,  
1 Homi Bhabha Road, Colaba, Mumbai, 400 005, India}
\email{amal@math.tifr.res.in}

\baselineskip=13.4044pt





\begin{abstract}
In this paper, we study the Atiyah-Segal completion problem for 
the equivariant algebraic $K$-theory.
We show that this completion problem  has a positive solution for the action 
of connected groups on smooth projective schemes. In contrast, we show that 
this problem has a negative solution for non-projective smooth schemes, even 
if the action has only finite stabilizers. 
\end{abstract}

\maketitle

\section{Introduction}\label{section:Intro}
The equivariant $K$-theory for topological spaces with group action was
invented by Atiyah before Quillen discovered the algebraic $K$-theory of
schemes. This theory had a significant impact on the subsequent works of
Atiyah, Segal and others, which included the celebrated index theorem of
Atiyah and Singer. This paper aims to study the Atiyah-Segal 
completion problem \cite{ASegal} for algebraic equivariant $K$-theory.

\subsection{The Atiyah-Segal Theorem}
For a compact Lie group $G$, let $R(G)$ denote the ring of
virtual representations of $G$ and let $I_G$ denote the augmentation ideal
given by the kernel of the map $\epsilon: R(G) \to \Z$ that takes
a virtual representation to its rank. 
For a compact Hausdorff topological space $X$ with $G$-action, let $K^G_*(X)$ 
denote the equivariant $K$-theory of $G$-equivariant complex
vector bundles on $X$. Let $\wh{{K^G_*(X)}_{I_G}}$ denote the $I_G$-adic
completion of the $R(G)$-module $K^G_*(X)$.

\enlargethispage{5pt}

To study the representations 
of $G$ in terms of the singular cohomology of its classifying space $B_G$,
Atiyah \cite{Atiyah} showed for a finite group $G$ that there is indeed 
a strong connection between $R(G)$ and the topological $K$-theory of $B_G$.
More precisely, if $\sK(B_G)$ denotes the inverse limit of the Grothendieck 
groups of complex vector bundles on the finite skeleta of $B_G$, 
then there is a natural isomorphism
\begin{equation}\label{eqn:Atiyah}
\wh{{R(G)}_{I_G}} \xrightarrow{\simeq} \sK(B_G).
\end{equation}

This result was extended to all compact Lie groups by Atiyah and
Hirzebruch \cite{AH}. This was subsequently reinterpreted 
by Atiyah and Segal \cite{ASegal} in terms of the following very general
statement
about the equivariant $K$-theory of compact Hausdorff topological spaces.
Let $E_G \to B_G$ denote the universal $G$-bundle and for a compact
$G$-space $X$, let $X_G$ denote {\sl Borel} space 
$X \stackrel{G}{\times} E_G$ (see \S~\ref{subsection:Free}).

\begin{thm}[Atiyah-Segal]\label{thm:ASMain}
Let $X$ be a compact $G$-space such that $K^G_*(X)$ is a finite $R(G)$-module.
Then the map $K^G_*(X) \to K_*(X_G)$ induces an isomorphism
\[
\wh{{K^G_*(X)}_{I_G}} \xrightarrow{\simeq} K_*(X_G). 
\]
\end{thm}

\subsection{The algebraic formulation}
The completion theorem for equivariant $K$-theory in a sense studies
the question of the impact of weak-equivalence between two $G$-spaces
(without reference to the group action) on their equivariant $K$-theory.
After the invention of $\A^1$-homotopy theory, this question has become
very pertinent in algebraic geometry.

The equivariant algebraic $K$-theory of schemes under the action of
group schemes was founded by Thomason \cite{Thomason3} using the ideas
of Quillen's $K$-theory of exact categories and Waldhausen's
$K$-theory of categories with cofibrations and weak equivalences. 
In order to formulate and study the analogue of the 
Atiyah-Segal completion problem for the algebraic $K$-theory, we need
algebraic objects which correspond to $X_G$.

One problem is that $X_G$ is not a scheme and one does not
know how to define $K$-theory of $X_G$. However, this problem now has a
solution, thanks to the invention of $\A^1$-homotopy theory of schemes by
Morel and Voevodsky \cite{MV}. These authors have constructed a category
of motivic spaces which includes all smooth schemes over a base scheme
as well as the colimits of such smooth schemes. They further show that
there is a generalized cohomology theory on the stable homotopy category
of motivic spaces which restricts to Thomason's $K$-theory for smooth schemes.

The second problem is that the completion theorem of Atiyah and Segal is based 
on a strong assumption that $K^G_*(X)$ is a finite $R(G)$-module and this 
assumption is very crucial in their proofs. 
Thomason \cite{Thomason4} was probably the first to consider the completion 
problem for the algebraic $K$-theory. But he was forced to work with the 
$K$-theory with finite coefficients and also had to invert the Bott element. 
The reason he had to do this
is that the Bott inverted algebraic $K$-theory with finite
coefficients has the above finiteness property. This allowed 
Thomason to use the techniques of Atiyah-Segal.
Such a finiteness assumption is almost never true (unless we work over
finite fields) for the algebraic $K$-theory, even for the algebraic 
$K$-theory of a point.

\subsection{The main results}
The aim of this paper is to solve the algebraic Atiyah-Segal completion
problem. We show using $\A^1$-homotopy theory 
that the algebraic analogue of topological spaces like $X_G$
do exist as motivic spaces and this allows one to study the 
Atiyah-Segal completion problem for schemes. We show 
that this completion problem has a positive solution
for smooth and projective schemes. 

In contrast, we show an unexpected result that 
this problem has a negative solution if we weaken
the projectivity assumption. We show that this negative solution can occur
even if one assumes that the underlying group action has finite stabilizers.
The main results of this paper roughly
look as follows. We shall state these results in more precise form and 
explain the underlying terms in the body of the text.

\begin{thm}\label{thm:MAIN-P}
Let $G$ be a connected split reductive group over a field $k$ acting 
on a smooth projective $k$-scheme $X$ and let $\wh{{K^G_*(X)}_{I_G}}$
denote the $I_G$-adic completion of the $R(G)$-module $K^G_*(X)$.
Then for every $p \ge 0$, there is an isomorphism
\begin{equation}\label{eqn:Alg-AS0*1-P}
\wh{{K^G_p(X)}_{I_G}} \xrightarrow{\simeq} K_p(X_G). 
\end{equation}

\end{thm}

Since a connected linear algebraic group in characteristic zero
has a Levi decomposition, the reductivity assumption is not necessary
in this case.

\vskip .3cm

For $p = 0$, the completion theorem holds in the most general situation
without any condition on $G$ and $X$.

\begin{thm}\label{thm:Alg-AS-0-P}
Let $G$ be a linear algebraic group over $k$ acting on a smooth scheme $X$. 
Then there is an isomorphism
\begin{equation}\label{eqn:Alg-AS0*-P}
\wh{{K^G_0(X)}_{I_G}} \xrightarrow{\simeq} K_0(X_G).
\end{equation}
\end{thm}

This result also holds for singular schemes 
(see Theorem~\ref{thm:Alg-AS-All}). In particular, it extends 
\cite[Theorem~3.1]{Totaro1} (proven for $\Spec(k)$) to all quasi-projective
$k$-schemes.

The equivariant Riemann-Roch theorem of Edidin and Graham 
\cite{EG1} is an immediate consequence of Theorem~\ref{thm:Alg-AS-0-P} 
and its generalization for singular schemes.
Another direct consequence of \thmref{thm:Alg-AS-0-P} is the 
following.

\begin{cor}\label{cor:Equiv-detect}
Let $G$ be a linear algebraic group over $k$ and let $f:X \to Y$ be a 
$G$-equivariant morphism of smooth schemes over $k$ with $G$-action.
Suppose that $f$ is an $\A^1$-weak equivalence 
(without reference to $G$-action), then the map
$f^*: \wh{{K^G_0(Y)}_{I_G}} \to \wh{{K^G_0(X)}_{I_G}}$ is an isomorphism.
\end{cor}

This result can be useful in showing the failure of $\A^1$-weak equivalence
between smooth schemes with group actions.

As another application of our completion theorems, we prove the
equivariant analogue of the Quillen-Lichtenbaum conjecture 
(see Theorem~\ref{thm:EQLC}).

\vskip .3cm

The following result provides a counterexample to the completion theorem
for non-projective schemes.

\begin{thm}\label{thm:Counter-E-P}
Let $G$ be a one-dimensional torus over $\C$ and let
$X$ be the quotient of $G$ by the subgroup $\mu_2$. Then 
\begin{enumerate}
\item 
For $p > 0$ odd, the map $\wh{{K^G_p(X)}_{I_G}} \to K_p(X_G)$ is 
an isomorphism.
\item
For $p > 0$ even, there is a short exact sequence
\[
0 \to \wh{{K^G_p(X)}_{I_G}} \to K_p(X_G) \to \Z_2 \to 0.
\]
\end{enumerate}
\end{thm}

\vskip .3cm

\subsection{An outline of the paper.}
A brief outline of this paper is as follows.
In \S~\ref{section:A1-homotopy}, we fix our notations and very briefly recall
the $\A^1$-homotopy theory of schemes. In 
\S~\ref{section:MBS}, we construct our
motivic Borel spaces associated to group actions on smooth schemes
and prove some of their properties. In \S~\ref{section:IOS}, we
discuss the algebraic $K$-theory of motivic Borel spaces. We also
introduce a variant, called $\sK$-theory, for these spaces. This is an algebraic
analogue of a similar topological object introduced by Atiyah
\cite{Atiyah}. 

In \S~\ref{section:EKT} and \S~\ref{section:EKT-S},
we prove our main technical results which yield decompositions of the
equivariant $K$-theory of filtrable schemes and that of the
$\sK$-theory of the associated Borel spaces. These results allow us
to prove the completion theorem for the torus action.
The equivariant $K$-theory of smooth schemes for the action of 
connected reductive groups is studied in \S~\ref{section:KBR}.
We define the Atiyah-Segal completion map and prove the completion theorem
in \S~\ref{section:ASM}. The following section deals with the completion
theorem for the Grothendieck group of equivariant bundles on all schemes.
In \S~\ref{section:Fail}, we compute the $K$-theory of the Borel spaces 
associated to some non-projective schemes. This allows us to show the failure 
of the completion theorem in such cases.

\section{Review of $\A^1$-homotopy theory of schemes}
\label{section:A1-homotopy}
In this section, we fix our notations and set up the machinery of 
$\A^1$-homotopy theory and the homotopy category of smooth schemes
as constructed in \cite{MV}. We shall later construct our motivic
Borel spaces as objects of this homotopy category.

\subsection{Notations and conventions}\label{subsection:Note}
Let $k$ be a field. 
We shall use \cite{MV} and \cite{Voev1} as our basic references
for $\A^1$-homotopy theory of smooth schemes over $k$.
We shall use the following notations throughout this text.

\begin{enumerate}
\item
${\rm Nis}/k :$ \ The Grothendieck site of smooth schemes of finite type over 
$k$ with the Nisnevich topology. 
\item
$\Spc(k) :$ \  The category of simplicial sheaves on ${\rm Nis}/k$
(also known as the category of motivic spaces).
\item
$\Spc_{\bullet}(k) :$ \  The category of pointed simplicial sheaves 
on ${\rm Nis}/k$ (also known as the category of pointed motivic spaces) .
\item
$\sH(k) :$ \ The unstable $\A^1$-homotopy category of motivic spaces.
\item
$\sH_{\bullet}(k) :$ \ The unstable $\A^1$-homotopy category of pointed 
motivic spaces.
\item
$\sS\sH(k) :$ \ The stable $\A^1$-homotopy category of pointed 
motivic spaces as defined, for example, in \cite{Voev1}. 
\end{enumerate}

Recall that a simplicial set (such as the usual simplicial set $\Delta[n]$)
is an object of $\Spc(k)$ as a constant simplicial sheaf. 
Let $T$ denote the pointed motivic space $(\P^1_k, \infty)$.
Recall that an object of $\sS\sH(k)$ is a $T$-spectrum over
$\Spc_{\bullet}(k)$. There is a functor $\Sigma^{\infty}_T: \sH_{\bullet}(k) \to
\sS\sH(k)$ which is given by $\Sigma^{\infty}_T(A) = \left(A , T \wedge A,
T^2 \wedge A, \cdots \right)$. This functor has a right adjoint which
gives the $0^{\rm th}$ level of a spectrum. In particular, given a motivic space 
({\sl e.g.},  a smooth scheme) $X$, we get 
$\Sigma^{\infty}_T X : = \Sigma^{\infty}_T(X_{+})$ in which
$X_{+}$ is obtained from $X$ by adjoining the base point $\Spec(k)$.
This gives functors $\sH(k) \xrightarrow{+} 
\sH_{\bullet}(k) \xrightarrow{\Sigma^{\infty}_T} \sS\sH(k)$.
Recall that $S_s$ and $S_t$ denote the pointed motivic spaces
$S^1$ and $(\G_m, 1)$, respectively and we define $\Sigma^{a, b} =
S^{a-b}_s \wedge S^b_t$.
Recall that $\sS\sH(k)$ is a triangulated category whose shift
functor is given by $S_s \wedge (-)$. 

\enlargethispage{50pt}

\subsection{Ind-objects and pro-objects in a category}
\label{subsection:IoS}
Given any category $\sC$, one can define the categories of pro-objects
(resp. ind-objects) in $\sC$. These objects consist of covariant 
(resp. contravariant)
functors from small cofiltering categories to $\sC$. A good exposition of this
can be found in \cite[\S~2]{Isaksen}.

In this text, we shall be concerned with only those pro and ind-objects which 
are indexed by the set of positive integers.
Thus, an ind-object $A = \{A_i, \alpha_i\}_{i \in \N^{+}}$ in a category $\sC$ is
a sequence $\{A_1 \xrightarrow{\alpha_1} A_2 \xrightarrow{\alpha_2} \cdots \}$ 
of objects in $\sC$. We shall denote such an ind-object by
the formal colimit $``{{\underset{i}\varinjlim}}" A_i$. 
The category of ind-objects
in $\sC$ will be denoted by ${\rm ind}\sC$.
A morphism $f: ``{{\underset{i}\varinjlim}}" A_i \to
``{{\underset{j}\varinjlim}}" B_j$ in ${\rm ind}\sC$ is an element of the set
${\underset{i}\varprojlim} {\underset{j}\varinjlim} \
\Hom_{\sC}(A_i, B_j)$. 

In particular, such a morphism $f$ is 
same as giving a function $\lambda: \N^{+} \to \N^{+}$ and a morphism
$f_i : A_i \to B_{\lambda(i)}$ in $\sC$ for each $i \ge 1$
such that for any $j \ge i$, there is some $l \ge \lambda(i), l \ge \lambda(j)$
so that the diagram
\begin{equation}\label{eqn:Ind-Obj}
\xymatrix@C2.9pc{
A_i \ar[d]_{p_{i, j}} \ar[r]^{f_i}  &  B_{\lambda(i)} \ar[dr]^{q_{\lambda(i), l}} \\
A_j \ar[r]_{f_j} & B_{\lambda(j)}  \ar[r]_{q_{\lambda(j), l}} & B_l}
\end{equation}
commutes in $\sC$. We shall call such a morphism to be {\sl strict} if
$\lambda$ is the identity function. 
If $\sC$ is closed under filtered colimit, there is a realization functor
${\varinjlim}(-): {\rm ind}\sC \to \sC$ which is right adjoint to the
canonical embedding $\sC \to {\rm ind}\sC$. 
If $\sC$ is an abelian category, then so is the category ${\rm ind}\sC$
({\sl cf.} \cite[Theorem~8.6.5]{KS}).

By dualizing the notion of 
ind-objects, we get the category of pro-objects ${\rm pro}\sC$ in
$\sC$. In particular, a pro-object $\{A_i, \alpha_i\}_{i \in \N^{+}}$
is a sequence $\{A_1 \xleftarrow{\alpha_1} A_2 \xleftarrow{\alpha_2} 
\cdots \}$ of objects in $\sC$. 
A pro-object will be formally denoted by 
$``{{\underset{i}\varprojlim}}" A_i$. 
A morphism $f: ``{{\underset{i}\varprojlim}}" A_i \to
``{{\underset{j}\varprojlim}}" B_j$ in ${\rm pro}\sC$ is an element of the set
${\underset{j}\varprojlim} {\underset{i}\varinjlim} \
\Hom_{\sC}(A_i, B_j)$. 
If $F: \sC \to \sD$ is a contravariant
functor, then it induces a functor ${\rm ind}\sC \to {\rm pro}\sD$. 
This obvious fact will be used frequently in this text.

If $\sC$ has cofiltered limit,
the limit of $``{{\underset{i}\varprojlim}}" A_i$ will be denoted by
$\underset{i}\varprojlim \ A_i$.   
If $\sC$ is an abelian category, then so is ${\rm pro}\sC$.
If $f : ``{{\underset{i}\varprojlim}}" A_i \to
``{{\underset{i}\varprojlim}}" B_i$ is a strict morphism, then one checks
easily that ${\rm Ker}(f) = ``{{\underset{i}\varprojlim}}" {\rm Ker}(f_i)$
and ${\rm Coker}(f) = ``{{\underset{i}\varprojlim}}" {\rm Coker}(f_i)$.
In particular, a sequence of strict morphisms of pro-objects 
\begin{equation}\label{eqn:pro-exact}
``{{\underset{i}\varprojlim}}" A_i \to
``{{\underset{i}\varprojlim}}" B_i \to
``{{\underset{i}\varprojlim}}" C_i 
\end{equation}
is exact in the abelian category ${\rm pro}\sC$ if it restricts to an exact 
sequence of objects in $\sC$ for each $i \in \N^{+}$.
We should warn that the exactness of ~\eqref{eqn:pro-exact} does not 
imply that the sequence 
remains exact if we replace $``{{\underset{i}\varprojlim}}"$ by 
$\underset{i}\varprojlim$. 
We refer the reader to \cite[Appendix~4]{AM} for these facts about pro-objects
in abelian categories.
\subsubsection{The Mittag-Leffler condition}
Let $R$ be a unital commutative ring and let $\sC$ be the category of 
$R$-modules. Recall that an inverse system $\{A_i, \alpha_i\}_{i \in \N^{+}}$ of
$R$-modules is said to satisfy the {\sl Mittag-Leffler} condition if
for every $i \ge 1$, there is $j \ge i$ such that
${\rm Image}(A_{j'} \to A_i) = {\rm Image}(A_j \to A_i)$ for all
$j' \ge j$. 

The functor
$\underset{i}\varprojlim: {\rm pro}\sC \to \sC$ is left exact 
and its right derived functor 
${\bf R}^1\underset{i}\varprojlim$ is often denoted by
$\underset{i}\varprojlim ^1$. 
The following well known sufficient condition
for the vanishing of $\underset{i}\varprojlim ^1 \ A_i$ will be
used frequently in this text.

\begin{prop}\label{prop:ML-lim-1}
Let $\{A_i, \alpha_i\}_{i \in \N^{+}}$ be an inverse system of $R$-modules
which satisfies the Mittag-Leffler condition. Then 
$\underset{i}\varprojlim ^1 \ A_i = 0$.
\end{prop}

\subsection{Group action and quotients}\label{subsection:Free}
For the rest of this text, we shall consider only those schemes which are
quasi-projective over $k$. The category of such schemes will be denoted by  
$\Sch_k$ and an object of this category will often be called a $k$-scheme.
Let $\Sm_k$ denote the subcategory of smooth schemes in $\Sch_k$.

A {\sl linear algebraic group} $G$ over $k$ will mean a smooth and affine 
group scheme over $k$. A closed subgroup of $G$ will mean a closed immersion  
of linear algebraic groups over $k$. Recall from
\cite[Proposition~1.10]{Borel} that a linear algebraic group over $k$ is a
closed subgroup of a general linear group, defined over $k$.
Let $\Sch^G_k$ (resp. $\Sm^G_k$) denote the category of quasi-projective 
(resp. smooth and quasi-projective) $k$-schemes with $G$-action and 
$G$-equivariant maps.
An object of $\Sch^G_k$ will be often called a $G$-scheme.

Recall that an action of a linear algebraic group $G$ on a $k$-scheme $X$ 
is said to be {\sl linear} if $X$ admits a 
$G$-equivariant ample line bundle, a condition which is always satisfied
if $X$ is normal ({\sl cf.} \cite[Theorem~2.5]{Sumihiro} for $G$ connected
and \cite[5.7]{Thomason1} for $G$ general). All $G$-actions in this paper
will be assumed to be linear.

For $X, Y \in \Sch^G_k$, let
$X \stackrel{G} {\times} Y$ denote the quotient of the scheme $X \times Y$
by the diagonal action of $G$. This quotient is only an {\'e}tale sheaf
on $\Sch_k$ in general and not a scheme. However, we shall need to know
in this text that the quotients of this kind exist as schemes in the 
following cases ({\sl cf.} \cite[Proposition~23]{EG}). We refer to
\cite[Lemma~2.1]{Krishna2} for a proof. 
\begin{lem}\label{lem:sch}
Let $H$ be a linear algebraic group acting freely and linearly on a 
$k$-scheme $U$ such that the quotient $U/H$ exists as a quasi-projective
scheme. Let $X$ be a $k$-scheme with a linear action of $H$.
Then the mixed quotient $X \stackrel{H} {\times} U$ for the 
diagonal action on $X \times U$ exists as a scheme and is quasi-projective.
Moreover, this quotient is smooth if both $U$ and $X$ are so.
In particular, if $H$ is a closed subgroup of a linear algebraic group $G$ 
and $X$ is a $k$-scheme with a linear action of $H$, then the quotient 
$G \stackrel{H} {\times} X$ is a quasi-projective scheme.
\end{lem}

In this text, $\Sch^G_{{free}/k}$ will denote the full subcategory of $\Sch^G_k$
whose objects are those schemes $X$ on which $G$ acts freely such that
the quotient $X/G$ exists and is quasi-projective over $k$. 
The full subcategory of $\Sch^G_{{free}/k}$ consisting of smooth schemes will
be denoted by $\Sm^G_{{free}/k}$. The previous
result shows that if $U \in \Sch^G_{{free}/k}$, then $X \times U$
is also in $\Sch^G_{{free}/k}$ for every $G$-scheme $X$.

\section{Motivic Borel spaces}\label{section:MBS}
In this and the next section, we shall construct the two missing objects: the
motivic Borel spaces and the $\sK$-theory of Borel spaces. Together with
the equivariant algebraic $K$-theory, they form the main players in the
algebraic Atiyah-Segal completion problem.
We begin with the construction of the Motivic Borel spaces.
Some of these Borel type constructions can also be found in
\cite[Chapter~5]{Herman} (see also \cite{Deligne}, \cite{HKO}, and \cite{Tan}).

\subsection{Admissible gadgets}\label{subsection:Agdt}
Let $G$ be a linear algebraic group over $k$. 
All representations of $G$ in this text will be assumed to be
finite-dimensional. We shall say that a pair $(V,U)$ of smooth schemes over $k$
is a {\sl good pair} for $G$ if $V$ is a $k$-rational representation of $G$ and
$U \subsetneq V$ is a $G$-invariant open subset which is an object of
$\Sch^G_{{free}/k}$. It is known ({\sl cf.} \cite[Remark~1.4]{Totaro1}) that a 
good pair for $G$ always exists.

\begin{defn}\label{defn:Add-Gad}
A sequence of pairs $\rho = {\left(V_i, U_i\right)}_{i \ge 1}$ of 
$k$-schemes is called an {\sl admissible gadget} for $G$, if there exists a
good pair $(V,U)$ for $G$ such that $V_i = V^{\oplus i}$ and $U_i \subsetneq V_i$
is $G$-invariant open such that the following hold for each $i \ge 1$.
\begin{enumerate}
\item
$\left(U_i \oplus V\right) \cup \left(V \oplus U_i\right)
\subseteq U_{i+1}$ as $G$-invariant open subsets.
\item
$\codim_{U_{i+2}}\left(U_{i+2} \setminus 
\left(U_{i+1} \oplus V\right)\right) > 
\codim_{U_{i+1}}\left(U_{i+1} \setminus \left(U_{i} \oplus V\right)\right)$.
\item
$\codim_{V_{i+1}}\left(V_{i+1} \setminus U_{i+1}\right)
> \codim_{V_i}\left(V_i \setminus U_i\right)$.
\item
$U_i \in \Sm^G_{{free}/k}$.
\end{enumerate}
\end{defn}

The above definition is a variant of the  notion
of admissible gadgets in \cite[\S 4.2]{MV}, where these terms are
defined for vector bundles over a scheme. 
An example of an admissible gadget for $G$ can be constructed as follows. 
Choose a faithful $k$-rational representation $W$ of $G$ of dimension $n$. Then 
$G$ acts freely on an open subset $U$ of 
$V = W^{\oplus n} \simeq {\rm End}_k(W)$. 
Let $Z = V \setminus U$.
We now take 
$V_i = V^{\oplus i}, U_1 = U$ and  $U_{i+1} = \left(U_i \oplus V\right) \cup
\left(V \oplus U_i\right)$ for $i \ge 1$. Setting 
$Z_1 = Z$ and $Z_{i+1} = U_{i+1} \setminus \left(U_i \oplus V\right)$ for 
$i \ge 1$, one checks that $V_i \setminus U_i = Z^i$ and
$Z_{i+1} = Z^i \oplus U$.
In particular, $\codim_{V_i}\left(V_i \setminus U_i\right) =
i (\codim_V(Z))$ and 
$\codim_{U_{i+1}}\left(Z_{i+1}\right) = (i+1)d - i(\dim(Z))- d =  
i (\codim_V(Z))$,
where $d = \dim(V)$. Moreover, $U_i \to {U_i}/G$ is a principal $G$-bundle
of smooth schemes.

Given an {\sl admissible gadget} $\rho$ for $G$ and $X \in \Sm^G_k$, let
$X^i_G(\rho)$ denote the mixed quotient space $X \stackrel{G} {\times} U_i$.
We shall often write $X^i_G(\rho)$ as simply $X^i(\rho)$ if the underlying
group $G$ is understood in a context. 
The locally closed immersion $U_i = U_i \times \{0\} \inj U_{i+1}$ yields an
ind-object $``{{\underset{i}\varinjlim}}" X^i(\rho)$ in
$\Sm_k$. Such an object will also be called an ind-scheme.
Since $\Spc(k)$ admits all filtered colimits, we see that
$X_G(\rho) = X(\rho) := {\underset{i}\varinjlim} \ X^i(\rho)$
is a motivic space. Our construction of motivic Borel spaces is based on 
the following result.

\begin{prop}\label{prop:Rep-ind}
Let  $\rho$ and $\rho'$ be two admissible gadgets for $G$. Given any
$X \in \Sm^G_k$, there is a canonical motivic weak equivalence $X(\rho) 
\simeq X(\rho')$.
\end{prop} 
\begin{proof}
This was proven by Morel-Voevodsky \cite[Proposition~4.2.6]{MV}
when $X = \Spec(k)$ and a similar argument works in the general case as well.

For $i, j \ge 1$, we consider the smooth scheme $\sV_{i,j} =
{\left(X \times U_i \times V'_j\right)}/G$ and the open 
subscheme $\sU_{i,j} = {\left(X \times U_i \times U'_j\right)}/G$.
For a fixed $i \ge 1$, this yields a sequence ${\left(\sV_{i,j}, 
\sU_{i,j}, f_{i,j}\right)}_{j \ge 1}$, where $\sV_{i,j} \xrightarrow{\pi_{i,j}}
X^i(\rho)$ is a vector bundle, $\sU_{i,j} \subseteq \sV_{i,j}$ is an open
subscheme of this vector bundle and $f_{i,j} : \left(\sV_{i,j}, \sU_{i,j}\right)
\to \left(\sV_{i,j+1}, \sU_{i,j+1}\right)$ is the natural map of pairs
of smooth schemes over $X^i(\rho)$. Then ${\left(\sV_{i,j}, 
\sU_{i,j}, f_{i,j}\right)}_{j \ge 1}$ is an admissible gadget over 
$X^i(\rho)$ in the sense of \cite[Definition~4.2.1]{MV}.
Setting $\sU_{i} = \colim_j \ \sU_{i,j}$ and $\pi_i = \colim_j \ \pi_{i,j}$, it 
follows from [{\sl loc. cit.}, Proposition~4.2.3] that the map
$\sU_i  \xrightarrow{\pi_i} X^i(\rho)$ is an $\A^1$-weak equivalence. 

Taking the colimit of these maps as $i \to \infty$ and using [{\sl loc. cit.},
Corollary~1.1.21], we conclude that the map
$\sU \xrightarrow{\pi}  X(\rho)$ is an $\A^1$-weak equivalence,
where $\sU = \colim_{i, j} \ \sU_{i,j}$. Reversing the roles of $\rho$ and
$\rho'$, we find that the obvious map $\sU \xrightarrow{\pi'}  X(\rho')$ is 
also an $\A^1$-weak equivalence. This yields the canonical
isomorphism $\pi' \circ \pi^{-1} :X(\rho) \xrightarrow{\simeq}
X(\rho')$ in $\sH(k)$.
\end{proof}

\subsubsection{Admissible gadgets associated to a given $G$-scheme}
\label{subsubsection:Add-Sc}
A careful reader may have observed in the proof of 
Proposition~\ref{prop:Rep-ind} 
that we did not really use the fact that $G$ acts freely on the open subset 
$U_i$ (resp. $U'_j$) of the $G$-representation $V_i$ (resp. $V'_j$). One only 
needs to know that for each $i, j \ge 1$, the quotients $(X \times U_i)/G$
and $(X \times U'_j)/G$ are smooth schemes and the maps
${\left(X \times U_i \times V'_j\right)}/G \to (X \times U_i)/G$ and
${\left(X \times V_i \times U'_j\right)}/G \to (X \times U'_j)/G$
are vector bundles with appropriate properties.
This observation leads us to the following variant of
Proposition~\ref{prop:Rep-ind} which will be useful for us on some
occasions.

Let $G$ be a linear algebraic group over $k$ and let $X \in \Sch^G_k$.
We shall say that a pair $(V,U)$ of smooth schemes over $k$
is a good pair for the $G$-action on $X$, if $V$ is a $k$-rational 
representation of $G$ and $U \subseteq V$ is a $G$-invariant open subset such
that $X \times U$ is an object of $\Sch^G_{{free}/k}$.
We shall say that a sequence of pairs 
$\rho = {\left(V_i, U_i\right)}_{i \ge 1}$ of smooth schemes over $k$
is an {\sl admissible gadget for the $G$-action on $X$}, if there exists a
good pair $(V,U)$ for the $G$-action on $X$ such that $V_i = V^{\oplus i}$ and 
$U_i \subseteq V_i$ is $G$-invariant open subset such that the following hold 
for each $i \ge 1$.
\begin{enumerate}
\item
$\left(U_i \oplus V\right) \cup \left(V \oplus U_i\right)
\subseteq U_{i+1}$ as $G$-invariant open subsets.
\item
$\codim_{U_{i+2}}\left(U_{i+2} \setminus 
\left(U_{i+1} \oplus V\right)\right) > 
\codim_{U_{i+1}}\left(U_{i+1} \setminus \left(U_{i} \oplus V\right)\right)$.
\item
$\codim_{V_{i+1}}\left(V_{i+1} \setminus U_{i+1}\right)
> \codim_{V_i}\left(V_i \setminus U_i\right)$.
\item
$X \times U_i \in \Sch^G_{{free}/k}$.
\end{enumerate}

Notice that an admissible gadget for $G$ as in 
Definition~\ref{defn:Add-Gad} is an admissible gadget
for the $G$-action on every $G$-scheme $X$.

\begin{prop}\label{prop:Rep-ind-V}
Let $\rho_X$ and $\rho'_X$ be two admissible gadgets for the $G$-action on 
a smooth scheme $X$.
Then there is a canonical motivic weak equivalence 
\[
{\underset{i}\varinjlim} \
(X\stackrel{G}{\times}U_i) \ \simeq \
{\underset{j}\varinjlim}  \ (X\stackrel{G}{\times}U'_j).
\] 
\end{prop}

\begin{defn}\label{defn:Borel-Spc}
Given $X \in \Sm^G_k$,  the associated {\sl motivic Borel space}
is a motivic space $X_G$ of the form
$X_G(\rho)$, where $\rho = (V_i, U_i)_{i \ge 1}$ is an admissible
gadget for $G$. It follows from Proposition~\ref{prop:Rep-ind} that
$X_G$ is a well-defined object of $\sH(k)$. 

The motivic Borel space associated to $X = \Spec(k)$ will be called the 
{\sl classifying space} of $G$ and will be denoted by $B_G$. 
\end{defn}

\subsection{Morita equivalence for Borel spaces}\label{subsection:MOR}
The motivic Borel spaces corresponding to different algebraic groups
satisfy the following $\A^1$-homotopy version of the Morita equivalence.

\begin{prop}\label{prop:Morita0}
Let $H$ be a closed normal subgroup of a linear algebraic group $G$ and let
$F = G/H$. Let $f: X \to Y$ be a morphism in $\Sm^G_k$ 
which is an $H$-torsor for the restricted action. Then there is an
isomorphism $X_G \simeq Y_F$ in $\sH(k)$.
\end{prop}
\begin{proof}
We first observe from \cite[Corollary~12.2.2]{Springer} that
$F$ is also a linear algebraic group over the given ground field $k$.
Let $\rho = (V_i, U_i)_{i \ge 1}$ be an admissible gadget for $F$. The 
morphism $G \to F$ makes each $V_i$ a $k$-rational representation
of $G$ such that the open subset $U_i$ is $G$-invariant,
even though $G$ may not act freely on $U_i$.
In particular, $G$ acts on the product $X \times U_i$ via the diagonal
action.
Since $H$ acts freely on $X$ and $F$ acts freely on $U_i$, 
it follows that the map $X \times U_i \to X \stackrel{G}{\times} U_i$
is a $G$-torsor and hence $\rho =(V_i, U_i)_{i \ge 1}$ is an admissible gadget
for the $G$-action on $X$.

Since the map $X \stackrel{G}{\times} U_i \to 
Y \stackrel{F}{\times} U_i$ is an isomorphism for every $i \ge 1$,
we conclude from Proposition~\ref{prop:Rep-ind-V} that
$X_G \simeq \colim_i \ (X\stackrel{G}{\times}U_i) 
\xrightarrow{\simeq} Y_F$ in $\sH(k)$.
\end{proof}

\begin{cor}[Morita isomorphism]\label{cor:Morita1}
Let $H$ be a closed subgroup of a linear algebraic group $G$ and let
$X \in \Sm^H_k$. Let $Y$ denote the space $X \stackrel{H}{\times} G$
for the action $h \cdot (x, g) = (h\cdot x, gh^{-1})$. 
Then there is an isomorphism $X_H \simeq Y_G$ in $\sH(k)$.
\end{cor}
\begin{proof}
Define an action of $H \times G$ on $X \times G$ by
\begin{equation}\label{eqn:MoritaI**1}
(h, g) \cdot (x, g') = \left(hx, gg'h^{-1}\right)
\end{equation}
and an action of $H \times G$ on $X$ by $(h, g) \cdot x = hx$. 
Then the projection map $X \times G \xrightarrow{p} X$ is 
$\left(H \times G\right)$-equivariant and a 
$G$-torsor. Hence there is an isomorphism $X_H \simeq
(X \times G)_{H \times G}$ in $\sH(k)$ by Proposition~\ref{prop:Morita0}.

On the other hand, the projection map $X \times G 
\to X \stackrel{H}{\times} G$ is $\left(H \times G\right)$-equivariant and 
an $H$-torsor. Hence there is an isomorphism 
$(X \times G)_{H \times G} \simeq Y_G$ in $\sH(k)$ again by 
Proposition~\ref{prop:Morita0}. Combining these two isomorphisms, we get
$X_H \simeq Y_G$ in $\sH(k)$.
\end{proof}

Recall that a unipotent group $U$ over $k$ is called {\sl split} if it
has a filtration 
$\{e\} = U_0 \subseteq U_1 \subseteq \cdots \subseteq 
U_n = U$ by closed normal $k$-subgroups such that each
quotient group ${U_j}/{U_{j-1}}$ is isomorphic to the algebraic group $\G_a$
({\sl cf.} \cite[\S~3.4]{Springer}).

\begin{prop}\label{prop:Unip}
Let $G$ be a possibly non-reductive group over $k$.
Assume that $G$ has a Levi decomposition $G = L \ltimes G^u$ such
that $G^u$ is split over $k$ (e.g., if $k$ has characteristic zero).
Then for any $X \in \Sm^G_k$, the map $X_L \to X_G$ is an isomorphism
in $\sH(k)$.
\end{prop}
\begin{proof}
Let us denote the unipotent radical $G^u$ by $U$ and
let $\{e\} = U_0 \subseteq U_1 \subseteq \cdots \subseteq U_n = U$
the filtration of $U$ as above. Set $G_i = LU_i$ for $0 \le i \le n$.

We have a sequence of morphisms $X_L \xrightarrow{p_0} X_{G_1} 
\xrightarrow{p_1} \cdots \xrightarrow{p_{n-1}} X_G$ in ${\bf {Spc}}$ such that 
each $p_i$ is an $\A^1$-weak equivalence. Hence the map
$X_L \to X_G$ is an $\A^1$-weak equivalence. 
\end{proof}

\section{Algebraic $\sK$-theory of Borel spaces}
\label{section:IOS}
Having defined the motivic Borel spaces, our next objective is to define
the $K$-theory and $\sK$-theory of these spaces. In order to do so,
we need the following isomorphism of two ind-objects in the stable
homotopy category. Recall from \S~\ref{subsection:Note} that for 
$X \in \Spc(k)$, $\Sigma^{\infty}_T X$ denotes the $T$-spectrum 
$\Sigma^{\infty}_T(X_{+})$.

\begin{prop}\label{prop:Stable-Ind}
Let $\rho = \left(V_i, U_i\right)_{i \ge 1}$ and 
$\rho' = \left(V'_i, U'_i\right)_{i \ge 1}$ be two admissible gadgets for $G$.
Then for any  $X \in \Sm^G_k$ and $p \in \Z$, there is a canonical isomorphism 
\[
``{{\underset{i}\varinjlim}}" \ \Sigma^{\infty}_T X^i(\rho)[p]
\xrightarrow{\simeq}
``{{\underset{i}\varinjlim}}" \ \Sigma^{\infty}_T X^i(\rho')[p]
\]
of ind-objects in ${\sS \sH}(k)$.
\end{prop}
\begin{proof}
Since the shift functor $(-)[p]$ on ${\sS \sH}(k)$ commutes with 
$``{{\underset{i}\varinjlim}}"(-)$, it is enough to prove the
proposition without using the shift.
For $i, i' \ge 1$, we set 
$X^i_{i'} = {\left(X \times U_i \times U'_{i'}\right)}/G$.
For any $j \ge i, j' \ge i'$, there is a locally closed immersion
$\theta^{i,j}_{i',j'}: X^i_{i'} \to X^{j}_{j'}$ and this induces a natural
map $\sigma^{i,j}_{i',j'} = \Sigma^{\infty}_T\theta^{i,j}_{i',j'}: 
\Sigma^{\infty}_T X^i_{i'} \to \Sigma^{\infty}_T X^{j}_{j'}$ in ${\sS \sH}(k)$.
Set $\sU = \colim_i \ \colim_{i'} \ X^i_{i'}$ as a motivic space.
Since for any $i, i' \ge 1$, the inclusion $X^i_{i'} \inj \sU$ factors
through the inclusions $X^i_{i'} \inj X^l_{l} \inj \sU$ where
$l = {\rm max}(i, i')$, we see that $\sU$ is also the colimit
of the direct system of smooth schemes $\left\{\sU_{i}, \lambda_{i,j}\right\}$
if we let $\sU_i = X^i_i$ and $\lambda_{i,j} = \theta^{i,j}_{i,j}$.

There are natural projections $p_i: \sU_i \to X^i(\rho)$ and
$q_{i}: \sU_i \to X^i(\rho')$. These maps combine together to give us the
following morphisms of ind-objects in ${\sS \sH}(k)$:

\begin{equation}\label{eqn:Stable-Ind0}
``{{\underset{i}\varinjlim}}" \ \Sigma^{\infty}_T X^i(\rho) 
\xleftarrow{\wh{p}} 
``{{\underset{i}\varinjlim}}" \ \Sigma^{\infty}_T \sU_i \xrightarrow{\wh{q}}
``{{\underset{i}\varinjlim}}" \ \Sigma^{\infty}_T X^i(\rho').
\end{equation}
  
It suffices to show that these two morphisms of ind-objects are 
isomorphisms. We shall show that $\wh{p}$ is an isomorphism and exactly the
same proof works for $\wh{q}$.
Let $\gamma_{i,j}: X^i(\rho) \to X^j(\rho)$ denote the structure maps
of $``{{\underset{i}\varinjlim}}" \ X^i(\rho)$.

To show that $\wh{p}$ is an isomorphism in ${\rm ind}{\sS \sH}(k)$,
what we need to show is that for every $i \ge 1$, there exists $j =
\beta(i) \gg i$ and a morphism $\beta_i: \Sigma^{\infty}_T X^i(\rho) \to 
\Sigma^{\infty}_T \sU_j$
in ${\sS \sH}(k)$ such that the diagram

\begin{equation}\label{eqn:Stable-Ind1}
\xymatrix@C2.5pc{
\Sigma^{\infty}_T \sU_i \ar[r]^<<<<<{\Sigma^{\infty}_Tp_i} \ar[d] &  
\Sigma^{\infty}_T X^i(\rho) \ar[d] \ar[dl]_{\beta_i} \\
\Sigma^{\infty}_T \sU_j \ar[r]_<<<<<{\Sigma^{\infty}_Tp_j} & 
\Sigma^{\infty}_T X^j(\rho)}
\end{equation} 
commutes. 
In this case, $\beta: \N^{+} \to \N^{+}$ yields the inverse of
$\wh{p}$ in ${\rm ind}{\sS \sH}(k)$.

We have shown in Proposition~\ref{prop:Rep-ind} that the map
$p = \colim_i \ p_i: \sU \to X(\rho)$ of colimits is an $\A^1$-weak
equivalence. Since $\Sigma^{\infty}_T$ preserves colimits (the colimit of
$T$-spectra is taken level-wise), this implies in particular that the map
$\psi = \Sigma^{\infty}_T p : \colim_i \ \Sigma^{\infty}_T \sU_i \to 
\colim_i \ \Sigma^{\infty}_T X^i(\rho)$ is an isomorphism in ${\sS \sH}(k)$.
Let $\phi: \colim_i \ \Sigma^{\infty}_T X^i(\rho) \to 
\colim_i \ \Sigma^{\infty}_T \sU_i$ be the inverse of $\psi$.

Since each $\Sigma^{\infty}_T X^i(\rho)$ is a compact object of 
${\sS \sH}(k)$ ({\sl cf.} \cite[Proposition~5.5]{Voev1}), the composite map
$\Sigma^{\infty}_T X^i(\rho) \to \colim_i \ \Sigma^{\infty}_T X^i(\rho) 
\xrightarrow{\phi} \colim_i \ \Sigma^{\infty}_T \sU_i$ factors through
a map $\Sigma^{\infty}_T X^i(\rho) \xrightarrow{\beta_i} 
\Sigma^{\infty}_T \sU_{i'} \to \colim_i \ \Sigma^{\infty}_T \sU_i$ for 
some $i' \gg i$.
For every $j \ge i'$, let $\beta_{i,j}$ denote the composite map
\[
\beta_{i,j}: \Sigma^{\infty}_T X^i(\rho) \xrightarrow{\beta_i} 
\Sigma^{\infty}_T \sU_{i'} \xrightarrow{\Sigma^{\infty}_T \lambda_{i',j}} 
\Sigma^{\infty}_T \sU_{j}.
\]

We claim that the diagram
\begin{equation}\label{eqn:Stable-Ind2}
\xymatrix@C2.9pc{
& \Sigma^{\infty}_T X^i(\rho) \ar[dl]_{\beta_{i,j}} 
\ar[d]^{\Sigma^{\infty}_T \gamma_{i,j}} 
\\
\Sigma^{\infty}_T \sU_{j} \ar[r]_<<<<<<<{\Sigma^{\infty}_T p_j} & 
\Sigma^{\infty}_T X^j(\rho)}
\end{equation}
commutes for all $j \gg i$.
To prove this claim, we consider the bigger diagram
\begin{equation}\label{eqn:Stable-Ind3}
\xymatrix@C2.9pc{
\Sigma^{\infty}_T X^i(\rho) \ar[r]^{\beta_{i,j}} 
\ar[dr]_{\Sigma^{\infty}_T \gamma_{i,j}} & \Sigma^{\infty}_T \sU_{j} 
\ar[d]^{\Sigma^{\infty}_T p_j} \ar[r]^{\Sigma^{\infty}_T \lambda_j} & 
\Sigma^{\infty}_T \sU \ar[d]_{\simeq}^{\psi} \\
& \Sigma^{\infty}_T X^j(\rho) \ar[r]_{\Sigma^{\infty}_T \gamma_j} 
& \Sigma^{\infty}_T X(\rho).}
\end{equation}

For every $j \ge i'$, the square on the right clearly commutes and the 
outer trapezium commutes by the construction of $\beta_{i,j}$ and the
fact that ${\phi}^{-1} = \psi$.
In particular, the maps $\Sigma^{\infty}_T p_j \circ \beta_{i,j}$ and
$\Sigma^{\infty}_T \gamma_{i,j}$ become same when we go all the way to
the colimit $\Sigma^{\infty}_T X(\rho)$. 
Since $\Sigma^{\infty}_T X^i(\rho)$ is compact, the map
\[
\colim_{j \ge i} \ \Hom_{{\sS \sH}(k)}\left(\Sigma^{\infty}_T X^i(\rho),
\Sigma^{\infty}_T X^j(\rho)\right) \to
\Hom_{{\sS \sH}(k)}\left(\Sigma^{\infty}_T X^i(\rho), 
\Sigma^{\infty}_T X(\rho)\right)
\]
is an isomorphism of sets.
We conclude that the maps $\Sigma^{\infty}_T p_j \circ \beta_{i,j}$ and
$\Sigma^{\infty}_T \gamma_{i,j}$ must become same when $j \gg i'$.
This proves the claim.

Our next claim is that the diagram
\begin{equation}\label{eqn:Stable-Ind4}
\xymatrix@C2.9pc{
\Sigma^{\infty}_T \sU_{i} \ar[r]^<<<<<<<{\Sigma^{\infty}_T p_i} 
\ar[d]_{\Sigma^{\infty}_T \lambda_{i,j}} & 
\Sigma^{\infty}_T X^i(\rho) \ar[dl]^{\beta_{i,j}} \\
\Sigma^{\infty}_T \sU_{j}  & }
\end{equation}
commutes for all $j \gg i$.
To do this, we consider for every $j \ge i'$, the diagram
\begin{equation}\label{eqn:Stable-Ind5}
\xymatrix@C2.9pc{
\Sigma^{\infty}_T \sU_{i} \ar[d]_{\Sigma^{\infty}_T p_i} 
\ar[r]^{\Sigma^{\infty}_T \lambda_{i,j}} & \Sigma^{\infty}_T \sU_{j}  
\ar[r]^{\Sigma^{\infty}_T \lambda_j} \ar[d]^{\Sigma^{\infty}_T p_j} & 
\Sigma^{\infty}_T \sU \\
\Sigma^{\infty}_T X^i(\rho) \ar[ur]^{\beta_{i,j}} 
\ar[r]_{\Sigma^{\infty}_T \gamma_{i,j}} & 
\Sigma^{\infty}_T X^j(\rho) \ar[r]_{\Sigma^{\infty}_T \gamma_j} &
\Sigma^{\infty}_T X(\rho) \ar[u]_{\phi}^{\simeq}.} 
\end{equation}

By the construction of $\beta_{i,j}$, we know that
$\Sigma^{\infty}_T \lambda_j \circ \beta_{i,j} = 
\phi \circ \Sigma^{\infty}_T \gamma_i$. On the other hand, we also know that
\begin{equation}\label{eqn:Stable-Ind6}
\begin{array}{lll}
\psi \circ \Sigma^{\infty}_T \lambda_j \circ \Sigma^{\infty}_T \lambda_{i,j}
& = & \psi \circ \Sigma^{\infty}_T \lambda_i \\ 
& = & \Sigma^{\infty}_T \gamma_i \circ \Sigma^{\infty}_T p_i \\
& = & \Sigma^{\infty}_T \gamma_j \circ \Sigma^{\infty}_T \gamma_{i,j} \circ
\Sigma^{\infty}_T p_i.
\end{array}
\end{equation}

Equivalently, we get 
\[
\begin{array}{lll}
\Sigma^{\infty}_T \lambda_j \circ \beta_{i,j} \circ \Sigma^{\infty}_T p_i 
& = & \phi \circ \Sigma^{\infty}_T \gamma_i \circ \Sigma^{\infty}_T p_i \\
& = & \phi \circ \Sigma^{\infty}_T \gamma_j 
\circ \Sigma^{\infty}_T \gamma_{i,j} \circ \Sigma^{\infty}_T p_i \\
& = & \Sigma^{\infty}_T \lambda_j \circ \Sigma^{\infty}_T \lambda_{i,j},
\end{array}
\]
where the last equality follows from ~\eqref{eqn:Stable-Ind6}.
Since $\Sigma^{\infty}_T \sU_i$ is a compact object of ${\sS \sH}(k)$, 
the same
argument as above shows that we must have
$\beta_{i,j} \circ \Sigma^{\infty}_T p_i = \Sigma^{\infty}_T \lambda_{i,j}$
for all $j \gg i$. The two claims together prove ~\eqref{eqn:Stable-Ind1}
and hence the proposition.
\end{proof}

\subsection{Algebraic $K$-theory of schemes and spaces}
\label{subsection:KTSS}
For any $X \in \Sch_k$, let $G_*(X)$ (resp. $K_*(X)$) denote the Quillen 
$K$-theory of coherent sheaves (vector bundles) on $X$. One knows that
the functor $X \mapsto G_*(X)$ is covariant for proper maps and contravariant
for maps of finite Tor-dimension ({\sl cf.} \cite[\S~5.10]{Srinivas}). 
In particular, for finite Tor-dimension morphisms $X \xrightarrow{f} Y
\xrightarrow{g} Z$, one has

\begin{equation}\label{eqn:FTD}
(g \circ f)^* = f^* \circ g^* : G_*(Z) \to G_*(X).
\end{equation}

The functor $X \mapsto K_*(X)$ is contravariant for all
maps and covariant for proper maps of finite Tor-dimension.
Moreover, $K$-theory and $G$-theory satisfy the projection 
formulas whenever the pull-back and the push-forward maps are defined.

\subsubsection{$K$-theory of motivic spaces}
\label{subsubsection:KTMS}
Recall from \cite[\S 6.2]{Voev1} that there is a fibrant 
$T$-spectrum ${\bf{BGL}}$
in ${\sS \sH}(k)$ which represents the algebraic $K$-theory.
For any $A \in {\sS \sH}(k)$, one defines the {\sl algebraic $K$-theory}
of $A$ by 
\begin{equation}\label{eqn:KMS0}
K^{a,b}(A) = 
\Hom_{{\sS \sH}(k)}\left(A, \Sigma^{a,b}{\bf{BGL}}\right).
\end{equation}

For $X \in {\bf {Spc}}$, the algebraic $K$-theory 
$K^{a,b}(\Sigma^{\infty}_T X_{+})$
is denoted by $K^{a,b}(X)$.
Using the Bott-periodicity isomorphism ${\bf{BGL}} \simeq T \wedge {\bf{BGL}}$
({\sl cf.} \cite[Theorem~6.8]{Voev1}),
one finds that  
\begin{equation}\label{eqn:KMS}
\begin{array}{lll}
K^{a,b}(X) & = & 
\Hom_{{\sS \sH}(k)}\left(\Sigma^{\infty}_T X [2b-a], 
{\bf{BGL}}\right) \\
& = & 
\Hom_{{\sS \sH}(k)}\left(\Sigma^{\infty}_T (S^{2b-a}_s \wedge X_{+}), 
{\bf{BGL}}\right)
\end{array}
\end{equation}
for $X \in {\bf {Spc}}$. It is shown in \cite[Theorem~6.9]{Voev1} that the 
last term is same as the 
Quillen-Thomason $K$-theory $ K_{2b-a}(X)$, if $X$ is a smooth scheme of
finite type over $k$.  
This allows us to define the algebraic $K$-theory of a motivic space $X$
by 
\begin{equation}\label{eqn:KMS*}
K_p(X) := \Hom_{{\sS \sH}(k)}\left(\Sigma^{\infty}_TX[p],
{\bf{BGL}}\right).
\end{equation}
\enlargethispage{50 pt}
The result below lists some basic properties of the $K$-theory
of motivic spaces.

\begin{prop}\label{prop:Basic-K}
The assignment $X \mapsto K_*(X)$ is a contravariant functor on ${\sH}(k)$.
This coincides with the Quillen $K$-theory of algebraic vector bundles if
$X \in \Sm_k$. It has the following other properties.
\begin{enumerate}
\item
If $H \subseteq G$ is a closed subgroup of a linear algebraic group $G$ over $k$
and if $X \in \Sm^H_k$, then the map $K_*(Y_G) \to K_*(X_H)$ is an isomorphism, 
where $Y = X \stackrel{H}{\times} G$. 
\item 
For $X \in \Sm^G_{{free}/k}$, the map $K_*(X/G) \to K_*(X_G)$ is an isomorphism.
\item
If $G$ has a Levi decomposition $G = L \ltimes G^u$ such
that $G^u$ is split over $k$, then $K_*(X_L) \simeq K_*(X_G)$ for any
$X \in \Sm^G_k$.
\end{enumerate}
\end{prop}
\begin{proof}
The contravariance follows from the fact that $A \mapsto K_*(A)$ is
a generalized cohomology theory on ${\sS \sH}(k)$. The isomorphism
with Quillen $K$-theory for smooth schemes is proven in 
\cite[Theorem~6.9]{Voev1}. The property (1) follows from 
Corollary~\ref{cor:Morita1}. 
The property (3) follows from Proposition~\ref{prop:Unip}.

To prove (2), let $(V_i, U_i)_{i \ge 1}$ be an admissible gadget for $G$.
This gives a compatible sequence of maps $X \stackrel{G}{\times} U_i \inj
X \stackrel{G}{\times} V_i \to X/G$, where the first map is an open immersion
and the second map is a vector bundle projection. In particular, this map is
an $\A^1$-weak equivalence ({\sl cf.} \cite[Example~3.2.2]{MV}).
We conclude that the map $\colim_{i} (X \stackrel{G}{\times} V_i) \to X/G$
is an $\A^1$-weak equivalence.

On the other hand, it follows from the definition of an admissible gadget
in Definition~\ref{defn:Add-Gad} that map 
$\colim_{i} (X \stackrel{G}{\times} U_i) \to 
\colim_{i} (X \stackrel{G}{\times} V_i)$ is an isomorphism of motivic spaces.
It follows that the map $X_G \simeq \colim_{i} (X \stackrel{G}{\times} U_i)
\to X/G$ is an $\A^1$-weak equivalence and this proves (2).
\end{proof}

\begin{thm}\label{thm:Stable-Ind**}
Let $\rho$ and $\rho'$ be two admissible gadgets for $G$. Then for any  
$X \in \Sm^G_k$ and $p \ge 0$, there is a canonical isomorphism 
\begin{equation}\label{eqn:Stable-Ind**0}
``{{\underset{i}\varprojlim}}" \ K_p\left(X^i(\rho)\right) 
\xrightarrow{\simeq}
``{{\underset{i}\varprojlim}}" \ K_p\left(X^i(\rho')\right)
\end{equation}
of pro-abelian groups. In particular, 
\begin{enumerate}
\item
${\underset{i}{\varprojlim}^{m}} \ K_p\left(X^i(\rho)\right) \
\simeq \ {\underset{i}{\varprojlim}^{m}} \ K_p\left(X^i(\rho')\right)$
for $m \ge 0$.
\item
$``{{\underset{i}\varprojlim}}" \ K_p\left(X^i(\rho)\right)$ satisfies the 
Mittag-Leffler condition if and only if so does
$``{{\underset{i}\varprojlim}}" \ K_p\left(X^i(\rho')\right)$.
\end{enumerate}
\end{thm}
\begin{proof}
Since a contravariant functor on ${\sS \sH}(k)$ takes an isomorphism
of ind-objects to an isomorphism of pro-objects, the isomorphism
~\eqref{eqn:Stable-Ind**0} 
follows directly from Propositions~\ref{prop:Stable-Ind} and
~\ref{prop:Basic-K}. 
The second assertion follows from the 
isomorphism ~\eqref{eqn:Stable-Ind**0} and
the elementary fact that the derived limits of two isomorphic pro-abelian
groups are isomorphic ({\sl cf.} \cite[Corollary~7.3.7]{Pros}). 
Also, a pro-abelian group $``{{\underset{i}\varprojlim}}" \ A_i$ satisfies the
Mittag-Leffler condition if and only if it is isomorphic to a pro-abelian
group $``{{\underset{i}\varprojlim}}" \ B_i$ such that the map $B_j \to B_i$ 
is surjective for all $j \ge i$. 
\end{proof}

\begin{remk}\label{remk:Stable-Sing*}
The use of $\A^1$-homotopy theory techniques prevents us from extending 
Theorem~\ref{thm:Stable-Ind**} to singular schemes with $G$-action.
However, we shall use a different and more geometric argument in 
\S~\ref{section:NON-P} to prove this for the $G_0$-functor.
\end{remk}

\subsubsection{$\sK$-theory of Borel spaces}\label{subsubsection:Kth-Borel}
Atiyah \cite{Atiyah} (see also \cite[\S 4.6]{AH}) introduced the $\sK$-theory 
for infinite $CW$-complexes to study the connection between the 
representation ring and the cohomology of the classifying spaces of 
compact Lie groups.
This is defined in terms of the projective limit of the usual topological
$K$-theory of the various skeleta of the given $CW$-complex
and it plays a very important role in understanding the topological
$K$-theory of such complexes. The above Theorem allows us to define an 
algebraic analogue of this theory for motivic spaces $X_G$. 

\begin{defn}\label{defn:Limit-K}
Let $G$ be a linear algebraic group over $k$ and let $X \in \Sm^G_k$. We define
\[
\sK_p(X_G) \ : = \ {\underset{i}\varprojlim} \ 
K_p(X^i(\rho))
\]
where $\rho = \left(V_i, U_i\right)$ is any admissible gadget for $G$.
\end{defn}

It follows from Theorem~\ref{thm:Stable-Ind**} that $\sK_p(X_G)$ is
well-defined for every $p \ge 0$.
Since each $K_p(X^i(\rho))$ is an $R(G)$-module via the isomorphism 
$K_p(X^i(\rho)) \xrightarrow{\simeq} K^G_p(X \times U_i)$ and since
the maps $K^G_p(X \times U_{i+1}) \to K^G_p(X \times U_i)$ are
$R(G)$-linear (see \S~\ref{subsubsection:Eq-KT}), we see that 
$\sK_p(X_G)$ is the limit of an inverse system of $R(G)$-linear
maps. In particular, it is an $R(G)$-module.

Since the structure maps of the pro-$R(G)$-module 
$``{{\underset{i}\varprojlim}}" \ K_p\left(X^i(\rho)\right)$ are defined in 
terms of the pull-back maps, one checks easily from the various properties
of the ordinary algebraic $K$-theory that the functor 
$X \mapsto \sK_*(X_G)$ on $\Sm^G_k$ satisfies all the properties of an
oriented cohomology theory except possibly the localization sequence.
However, it is true that this theory satisfies the localization sequence
as well. A proof of this will appear elsewhere.

For any $X \in \Sm^G_k$, the inclusions $X^i_G(\rho) \inj X_G(\rho)$
induce a natural map
\begin{equation}\label{eqn:Limit-K0}
\tau^G_X: K_*(X_G) \to \sK_*(X_G)
\end{equation}
which is surjective. This can be seen from ~\eqref{eqn:KMS*} and the
resulting Milnor exact sequence ({\sl cf.} \cite[Proposition~7.3.2]{Hovey}) 
\begin{equation}\label{eqn:MES}
0 \to {\underset{i}{\varprojlim}^1} \
K_{p+1}\left(X^i(\rho)\right) \to K_p\left(X(\rho)\right) \to
{\underset{i}\varprojlim} \ K_{p}\left(X^i(\rho)\right) \to 0.
\end{equation}

\section{Equivariant $K$-theory of filtrable schemes}
\label{section:EKT}
In this section, we prove some decomposition theorems for the 
equivariant and ordinary $K$-theory of a certain class of
schemes which are called filtrable. These decomposition theorems
will be used to prove our main results when the underlying group
is a torus. We begin with a review of equivariant $K$-theory, which is
the third main object in our study of the completion problem.

\subsection{Equivariant $K$-theory}\label{subsubsection:Eq-KT}
We recall the equivariant $K$-theory of group scheme actions from
\cite{Thomason2}, \cite{Thomason3} and \cite{Thomason1}.
Let $G$ be a linear algebraic group over $k$. For $X \in \Sch^G_k$,
$G^G(X)$ (resp. $K^G(X)$) is the $K$-theory spectrum of the $G$-equivariant  
coherent sheaves (resp. vector bundles) on $X$. 
For $p \ge 0$, $G^G_p(X)$ and $K^G_p(X)$ denote the stable homotopy groups
of the corresponding spectra.
For $X \in \Sm^G_k$, the map $K^G(X) \to G^G(X)$
is a weak equivalence of spectra. For $X \in \Sch^G_{{free}/k}$, the natural 
map $G(X/G) \to G^G(X)$ is a
weak equivalence ({\sl cf.} \cite[\S 3.2]{EG1}).

$K^G(X)$ is a ring spectrum with respect to the tensor product of 
equivariant vector bundles. $G^G(X)$ is a module spectrum over 
$K^G(X)$ with respect to the tensor product of equivariant vector bundles and
coherent sheaves.
This module structure induces a functorial map
$K^G_p(X) \otimes_{\Z} G^G_q(X) \to G^G_{p+q}(X)$.
The ring $K^G_0(k)$ is also denoted by $R(G)$ and is called the representation
ring of $G$.
$G^G(X)$ is a contravariant
functor in the group $G$, it is contravariant with respect to equivariant
flat maps, and is covariant with respect to equivariant proper maps.
 
Any map $f : X \to Y$ in $\Sch^G_k$ induces a morphism 
$f^*: K^G(Y) \to K^G(X)$ in the category of ring spectra.
Since $G^G(X)$ is a module spectrum over $K^G(X)$,
we see that the map $f$ makes $G^G(X)$ into a module spectrum over $K^G(Y)$.
In particular, the structure map $X \to \Spec(k)$ makes
$G^G_p(X)$ an $R(G)$-module for all $p \ge 0$ and $K^G_0(X)$ an
$R(G)$-algebra. 
Moreover, $f^*: K^G_p(Y) \to K^G_p(X)$ and $f^*: G^G_p(Y) \to G^G_p(X)$
(if $f$ is flat) are $R(G)$-linear.
 
The projection formula holds for equivariant proper maps. That is,
for an equivariant proper map $f:X \to Y$, $f_* : G^G(X) \to G^G(Y)$ is
a morphism of $K^G(Y)$-module spectra. In particular, the maps
$f_* : G^G_p(X) \to G^G_p(Y)$ is also $R(G)$-linear for all $p \ge 0$.
We refer to \cite[\S~1]{Thomason2} to verify the above properties.

Let $\epsilon: R(G) \to \Z$ denote the augmentation map which takes
any virtual representation to its rank. The kernel $I_G$ of this map
is called the augmentation ideal of $G$.
It is known that $R(G)$ is a commutative noetherian ring if $G$ is split
reductive ({\sl cf.} \cite[Lemmas~3.9, 9.2]{Krishna1}).
The $I_G$-adic completion of the ring $R(G)$ will be denoted by
$\wh{R(G)}$.
The equivariant $K$-theory also satisfies the following Morita equivalence
by a result of Thomason.

\begin{thm}[{\cite[Theorem~3.8]{JK}}, 
{\cite[Theorem~1.10]{Thomason1}}]\label{thm:Morita}
Let $G$ be a linear algebraic group over $k$ and let $H$ be a closed subgroup 
of $G$. For any $X \in \Sch^H_k$, there is an isomorphism 
$G^G_*(X \stackrel{H}{\times} G) \xrightarrow{\simeq} G^H_*(X)$ 
of $R(G)$-modules which is natural in $X$.
If $G$ has a Levi decomposition $G = L \ltimes G^u$ such that
$G^u$ is $k$-split, then $G^G_*(X) \simeq G^L_*(X)$ for any
$X \in \Sch^G_k$. These are ring isomorphisms if $X \in \Sm^H_k$.
\end{thm}

\subsection{Filtrable schemes}\label{subsection:Filt}
Let $G$ be a linear algebraic group over $k$ and let $X \in \Sch^G_k$. 
Following \cite[\S~3]{Brion2}, we shall say that $X$ is 
$G$-{\sl filtrable} 
if the fixed point subscheme $X^G$ is smooth and projective, 
there is an ordering $X^G = \stackrel{n}{\underset{m=0}{\coprod}}
Z_m$ of the connected components of the fixed point subscheme, a 
filtration of $X$ by $G$-invariant closed subschemes
\begin{equation}\label{eqn:filtration-BB}
{\emptyset} = X_{-1} \subsetneq X_0 \subsetneq \cdots \subsetneq X_n = X
\end{equation}
with $Z_m \subseteq  W_m := (X_m \setminus X_{m-1})$ and maps 
${\phi}_m : W_m \to Z_m$ for $0 \le m \le n$ 
which are all $G$-equivariant vector bundles such that the inclusions
$Z_m \inj W_m$ are the 0-section embeddings.
It is important to note that
the closed subschemes $X_m$'s may not be smooth even if $X$ is so.
Observe also that if $X$ is $G$-filtrable, then each closed subscheme 
$X_m$ is also $G$-filtrable. If $X$ is a smooth $G$-filtrable scheme, the 
associated motivic Borel space $X_G$ will be called filtrable.

We shall say that a $k$-scheme $X$ is filtrable if there are closed subschemes 
$\{Z_0, \cdots , Z_m\}$ of $X$ which are connected, smooth and projective, a 
filtration of $X$ by closed subschemes
\begin{equation}\label{eqn:filtration-BBO}
{\emptyset} = X_{-1} \subsetneq X_0 \subsetneq \cdots \subsetneq X_n = X
\end{equation}
with $Z_m \subseteq  W_m := (X_m \setminus X_{m-1})$ and maps 
${\phi}_m : W_m \to Z_m$ for $0 \le m \le n$ 
which are all vector bundles such that the inclusions
$Z_m \inj W_m$ are the 0-section embeddings.
It is clear that a $G$-filtrable scheme is also filtrable.
We prove the following decomposition formula for the 
equivariant $G$-theory of filtrable schemes.

\begin{thm}\label{thm:filter-Gen}
Let $G$ be a linear algebraic group over $k$ and let $X \in \Sch^G_k$ be 
$G$-filtrable as in ~\eqref{eqn:filtration-BB}.
Then there is an $R(G)$-module decomposition
\begin{equation}\label{eqn:filtration-BB0}
 G^G_*(X) \xrightarrow{\simeq} \
\stackrel{n}{\underset{m=0}{\oplus}} G^G_*(Z_m).
\end{equation}
\end{thm}  
\begin{proof}
We prove the theorem by induction on the length of the filtration.
For $n = 0$, the inclusion $Z_0 \inj X_0$ is the
0-section embedding of the $G$-equivariant vector bundle
$X = X_0 \xrightarrow{{\phi}_0} Z_0$. Hence, the theorem follows from the
homotopy invariance. 

We now assume by induction that $1 \le m \le n$ and 
we have an $R(G)$-module decomposition 
\begin{equation}\label{eqn:split0}
G^G_*\left(X_{m-1}\right) \xrightarrow{\simeq} \
\stackrel{m-1}{\underset{j=0}{\oplus}} G^G_*\left(Z_j\right).
\end{equation}

The localization exact sequence for the inclusions $i_{m-1} : X_{m-1} \inj X_m$ 
and $j_{m} : W_{m} \inj X_m$ of the $G$-invariant closed and open 
subschemes yields a long exact sequence of $R(G)$-linear maps

\begin{equation}\label{eqn:split*1}
\cdots \to G^G_p\left(X_{m-1}\right) \xrightarrow{i_{(m-1)*}}
G^G_p\left(X_m\right) \xrightarrow{j^*_m}
G^G_p\left(W_m\right) \xrightarrow{\partial} G^G_{p-1}\left(X_{m-1}\right) \to
\cdots.
\end{equation}
Using ~\eqref{eqn:split0}, it suffices now to construct an
$R(G)$-linear splitting 
of the pull-back $j^*_m$ in order to prove the theorem.  

Let $V_m \subset W_m \times Z_m$ be the graph of the projection 
$W_m \xrightarrow{{\phi}_m} Z_m$ and let $Y_m$ 
denote the closure of $V_m$ in $X_m \times Z_m$.  
Then $Y_m$ is a $G$-invariant closed subset of $X_m \times Z_m$ and $V_m$
is $G$-invariant and open in $Y_m$. 
We consider the composite maps 
\begin{equation}\label{eqn:split01}
p_m : V_m \inj W_m \times Z_m \to W_m, \ \ 
q_m : V_m \inj W_m \times Z_m \to Z_m \ \ {\rm and} 
\end{equation}
\[
{\ov{p}}_m : Y_m \inj X_m \times Z_m \to X_m, \ \ 
{\ov{q}}_m : Y_m \to X_m \times Z_m \to Z_m 
\]
in $\Sch^G_k$. Note that ${\ov{p}}_m$ is a projective morphism since $Z_m$ is 
projective. The map $q_m$ is smooth and $p_m$ is an isomorphism. 

We consider the diagram 
\begin{equation}\label{eqn:split1}
\xymatrix{
{G^G_*\left(Z_m\right)} \ar[r]^{{\ov{q}}^*_m} 
\ar[d]_{{\phi}^*_m}^{\simeq} &
{G^G_*\left(Y_m\right)} \ar[d]^{{{\ov{p}}_m}_*} \\
{G^G_*\left(W_m\right)} & 
{G^G_*\left(X_m\right)} \ar[l]^{j^*_m}.}
\end{equation} 
The map ${\ov{q}}^*_m$ is the composite $K^G_*(Z_m) \to K^G_*(Y_m) \to 
G^G_*(Y_m)$, where one identifies $K^G_*(Z_m)$ and $G^G_*(Z_m)$ since $Z_m$ is
smooth. The map ${{\phi}^*_m}$ is an isomorphism by the homotopy invariance. 
It suffices to show that this diagram commutes. For, the map 
$s_m : = {{\ov{p}}_m}_* \circ {\ov{q}}^*_m \circ
{{\phi}^*_m}^{-1}$ will then give the desired splitting of the map
$j^*_m$. Note that $s_m$ is $R(G)$-linear since so are all the maps in
~\eqref{eqn:split1}.

We now consider the commutative diagram in $\Sch^G_k$:
\[
\xymatrix@C3pc{
X_m & W_m \ar[l]_{j_m}& \\
Y_m \ar[u]^{{\ov{p}}_m} \ar[dr]_{{\ov{q}}_m} & V_m \ar[u]_{p_m} \ar[d]^{q_m}
\ar[l]_{{\ov{j}}_m} & W_m \ar[ul]_{id} 
\ar[l]^{(id, {\phi}_m)} \ar[dl]^{{\phi}_m} \\
& Z_m. & }
\]
Since the top left square is Cartesian with ${\ov{p}}_m$ projective and 
$j_m$ an open immersion, it follows that $j^*_m \circ {{\ov{p}}_m}_*  = 
{p_m}_* \circ {\ov{j}}^*_m$. 
Now, using the fact that $(id, {\phi}_m)$ is an isomorphism, 
we get 
\[
\begin{array}{lllll}
j^*_m \circ {{\ov{p}}_m}_* \circ {\ov{q}}^*_m & = & 
{p_m}_* \circ {\ov{j}}^*_m \circ {\ov{q}}^*_m & = &  
{p_m}_* \circ q^*_m \\
& = & {p_m}_* \circ {(id, {\phi}_m)}_* \circ {(id, {\phi}_m)}^* \circ
q^*_m & = & {id}_* \circ {\phi}^*_m \\
& = &  {\phi}^*_m. & &  
\end{array} 
\]
This proves the commutativity of ~\eqref{eqn:split1} and hence the theorem.
\end{proof}

We end this section with the following elementary result which will be used 
often in this text without an explicit mention. 
For a regular closed immersion $Y \inj X$, let $N_{Y/X}$ denote the 
normal bundle of $Y$ under this embedding.
Consider a Cartesian diagram of schemes
\begin{equation}\label{eqn:Tor-ind-K0}
\xymatrix@C.5pc{
Y' \ar[r]^{f'} \ar[d]_{g'} & X' \ar[d]^{g} \\
Y \ar[r]_{f} & X.}
\end{equation}

We shall say that this square is transverse, if the horizontal
maps are regular closed immersions and the map $N_{{Y'}/{X'}} \to
(g')^*(N_{{Y}/{X}})$ is an isomorphism.
One example of transverse squares that we shall often encounter in this
text is when $f$ is the 0-section embedding of a vector bundle
$\phi: X \to Y$ and $X' = (g')^*(X)$.

\begin{lem}\label{lem:Tor-ind-K}
Given a commutative square as in ~\eqref{eqn:Tor-ind-K0}, the following are
true.
\begin{enumerate}
\item
If ~\eqref{eqn:Tor-ind-K0} is a Cartesian square in $\Sch^G_{{free}/k}$, then
the resulting square of quotients is also Cartesian. If any of the maps
in this diagram is proper (resp. flat), then the map of 
quotients is also proper (resp. flat).  
\item
If $g$ is proper and either $f$ is  
flat or the square is transverse, then it is Tor-independent.
In particular, one has $f^* \circ g_* = g'_* \circ f'^*$ in
$\Hom(G_*(X'), G_*(Y))$.
\end{enumerate}
\end{lem}
\begin{proof}
The first part is elementary and follows easily from the smooth descent
for proper and flat maps.
Recall that the square  ~\eqref{eqn:Tor-ind-K0} is Tor-independent if
the sheaf $\Tor^{\sO_X}_i(\sO_{X'}, \sO_Y) = 0$ for $i > 0$.
This is obviously true if $f$ is flat.
Suppose now that ~\eqref{eqn:Tor-ind-K0} is transverse, and
let $\sF_{\bullet} \to \sO_Y$ be the Koszul resolution of $\sO_Y$ as
an $\sO_X$-module. It is then easy to check using transversality that
$\sF_{\bullet} {\underset{\sO_X}\otimes}\ \sO_{X'} \to \sO_{Y'}$
is the Koszul resolution of $\sO_{Y'}$ as an $\sO_{X'}$-module and this
implies Tor-independence. The claim about the
equality of two composite maps in $G$-theory now follows from
\cite[Proposition~5.13]{Srinivas}. 
\end{proof}

\section{$\sK$-theory of filtrable Borel spaces}
\label{section:EKT-S}
Let $G$ be any linear algebraic group over $k$ and let 
$\rho = \left(V_i, U_i\right)_{i \ge 1}$ be an admissible gadget for $G$
given by a good pair $(V,U)$.
Then for any $X \in \Sch^G_k$ and $i \ge 1$, 
there are natural maps $X \stackrel{G}{\times} U_i \stackrel{s_i}{\inj} 
X \stackrel{G}{\times} (U_i \oplus V) \stackrel{t_i}{\inj} 
X \stackrel{G}{\times} U_{i+1}$, where the first map is the 0-section of a
vector bundle and the second map is an open immersion.
In particular, the map $f_{X,i} :X^i(\rho) \inj X^{i+1}(\rho)$
is the composition of a 0-section embedding and an open inclusion.

It follows that there is a pro-$R(G)$-module 
$``{{\underset{i}\varprojlim}}" \ G_p(X^i(\rho))$ with structure maps
$f^*_{X,i}: G_p(X^{i+1}(\rho)) \to G_p(X^i(\rho))$ for any $p \ge 0$.
Moreover, it follows from Lemma~\ref{lem:Tor-ind-K} that the functor $X \mapsto 
``{{\underset{i}\varprojlim}}" \ G_p(X^i(\rho))$ 
is covariant for proper maps and contravariant for flat
maps in $\Sch^G_k$. The  pro-$R(G)$-module  
$``{{\underset{i}\varprojlim}}" \ K_p(X^i(\rho))$ is contravariant
for all maps in $\Sch^G_k$ and covariant for smooth and proper maps.
We define the {\sl $\rho$-dependent} $\sK$-theory and $\sG$-theory by 
\begin{equation}\label{eqn:LocalSK}
\sG^{\rho}_p(X_G) \ := \ {\underset{i}\varprojlim} \ G_p\left(X^i(\rho)\right); 
\ \ \ 
\sK^{\rho}_p(X_G) \ := \ {\underset{i}\varprojlim} \ K_p\left(X^i(\rho)\right).
\end{equation}

We specify the admissible pair $\rho$ because
we do not know if these are independent of $\rho$ if $X$ is not smooth.
We see that the assignment $X \mapsto \sG^{\rho}_*(X_G)$ is covariant for
proper maps and contravariant for flat maps in $\Sch^G_k$.
We want to prove an analogue of Theorem~\ref{thm:filter-Gen} for
these groups.

Let $\rho = \left(V_i, U_i\right)$ be an admissible gadget for $G$ such that
each ${U_i}/G$ is projective. Let $X \in \Sch^G_k$ be a $G$-filtrable scheme 
with the filtration given by ~\eqref{eqn:filtration-BB}. We set 
\[
X^i = X \stackrel{G}{\times} U_i, \ X^i_m = X_m \stackrel{G}{\times} U_i, \
W^i_m = W_m \stackrel{G}{\times} U_i \ {\rm and} \ 
Z^i_m = Z_m \stackrel{G}{\times} U_i.
\]

Given the $G$-equivariant filtration of $X$ as in ~\eqref{eqn:filtration-BB}, 
it is easy to see that for each $i \ge 1$, there is an 
associated system of filtrations 
\begin{equation}\label{eqn:filter-Equiv1}
{\emptyset} = X_{-1}^i \subsetneq X_0^i \subsetneq \cdots \subsetneq X_n^i = X^i
\end{equation}
with $Z^i_m \subseteq W^i_m = (X_m^i \setminus X_{m-1}^i)$ and maps 
${\phi}^i_m : W_m^i \to Z_m^i$ for 
$0 \le m \le n$ which are all vector bundles. Moreover, 
as $G$ acts trivially on each $Z_m$, we have that 
$Z_m^i \simeq Z_m \times \left({U_i}/G\right)$ is smooth and projective
since $Z_m$ is smooth and projective. We conclude that the filtration
~\eqref{eqn:filter-Equiv1} of each $X^i$ satisfies all the conditions of 
Theorem~\ref{thm:filter-Gen}. In other words, each $X^i_m$ is filtrable.
In particular, there are split exact sequences 
\begin{equation}\label{eqn:ft*}
0 \to G_*\left(X_{m-1}^i\right) \to  G_*\left(X_{m}^i\right)
\to  G_*\left(W_{m}^i\right) \to 0
\end{equation}
for all $0 \le m \le n$ and $i \ge 1$ by Theorem~\ref{thm:filter-Gen}.
Let $Y_m \inj X_m \times Z_m$ be as in ~\eqref{eqn:split01}.
One can check that the map $Y^i_m \to X^i_m$ is projective ({\sl cf.}
\cite[Lemma~5.1]{Krishna3}).

\enlargethispage{50pt}
\begin{lem}\label{lem:filter-commute}
Let $X \in \Sch^G_k$ be a $G$-filtrable scheme. Then for all $0 \le m \le n$ and
$i \ge 1$, the diagram
\begin{equation}\label{eqn:filt-comm1}
\xymatrix@C.5pc{
{G_*\left(Z_m^{i+1}\right)} 
\ar@/^1pc/[rr]^{{\ov{q}}^*_m} 
\ar[dd]_{{\phi}^*_m}^{\simeq} \ar[dr]^{f^*_{Z_m,i}} & &
{G_*\left(Y^{i+1}_m\right)} \ar[dd]^{{{\ov{p}}_m}_*} \ar[dr]^{f^*_{Y_m,i}} & \\
& {G_*\left(Z_m^i\right)} \ar@/_1pc/[rr]_<<<<<<<<{{\ov{q}}^*_m} 
\ar[dd]_{{\phi}^*_m}^{\simeq} & & {G_*\left(Y_m^i\right)} 
\ar[dd]^{{{\ov{p}}_m}_*} \\
{G_*\left(W_m^{i+1}\right)} \ar[dr]_{f^*_{W_m,i}} & &
{G_*\left(X^{i+1}_m\right)} \ar@/^1pc/[ll]^<<<<{j^*_m} \ar[dr]^{f^*_{X_m,i}} & \\
& {G_*\left(W_m^i\right)} & & {G_*\left(X^i_m\right)}
\ar@/^1pc/[ll]^{j^*_m}}
\end{equation} 
commutes.
\end{lem}
\begin{proof}
We have shown in the proof of Theorem~\ref{thm:filter-Gen}
({\sl cf.} ~\eqref{eqn:split1}) that the front and the back squares commute.
The right, left and the bottom squares commute by the covariant and
contravariant functoriality of the inverse systems $\left\{G_*(X^i_m)\right\}$
({\sl cf.} Lemma~\ref{lem:Tor-ind-K}).
The top square commutes by the contravariant property of
the inverse systems $\left\{K_*(X^i_m)\right\}$ and the functorial
isomorphism $K_*(Z^i_m) \xrightarrow{\simeq} G_*(Z^i_m)$.
\end{proof}

Let $X \in \Sch^G_k$ be as above and consider the commutative diagram
({\sl cf.} Lemma~\ref{lem:Tor-ind-K})
\begin{equation}\label{eqn:filt-comm4}
\xymatrix@C1.8pc{
0 \ar[r] & G_*\left(X^{i+1}_{m-1}\right) 
\ar[r]^{\left({\iota^{i+1}_{m-1}}\right)_*} \ar[d]_{f^*_{X_{m-1},i}} &
G_*\left(X^{i+1}_m\right) \ar[r]^{\left({j^{i+1}_{m}}\right)^*} 
\ar[d]^{f^*_{X_{m},i}} & G_*\left(W^{i+1}_m\right) \ar[d]^{f^*_{W_{m},i}} \ar[r] & 0 \\
0 \ar[r] & G_*\left(X^{i}_{m-1}\right) 
\ar[r]_{\left({\iota^{i}_{m-1}}\right)_*} &
G_*\left(X^{i}_m\right) \ar[r]_{\left({j^{i}_{m}}\right)^*} 
& G_*\left(W^{i}_m\right) \ar[r] & 0} 
\end{equation}
of split short exact sequences ~\eqref{eqn:ft*}, where the map
${\left({j^{i}_{m}}\right)^*}$
is split by $s^i_m : = \left({\ov{p}}^i_m\right)_* \circ 
\left({\ov{q}}^i_m\right)^* \circ \left(\left(\phi^i_m\right)^*\right)^{-1}$
for each $i \ge 1$ ({\sl cf.} diagram~\eqref{eqn:filt-comm1}). We now show that 
\begin{lem}\label{lem:filt-comm5}
$s^i_m \circ f^*_{W_m,i} =  f^*_{X_m,i} \circ s^{i+1}_m$.
\end{lem}
\begin{proof}
This is equivalent to showing that
\begin{equation}\label{eqn:filt-comm6}
 f^*_{X_m,i} \circ \left({\ov{p}}^{i+1}_m\right)_* \circ 
\left({\ov{q}}^{i+1}_m\right)^* = 
\left({\ov{p}}^i_m\right)_* \circ 
\left({\ov{q}}^i_m\right)^* \circ \left(\left(\phi^i_m\right)^*\right)^{-1} 
\circ  f^*_{W_m,i} \circ \left(\phi^{i+1}_m\right)^*.
\end{equation}

However, it follows from Lemma~\ref{lem:filter-commute} that
\[
\left(\left(\phi^i_m\right)^*\right)^{-1} 
\circ f^*_{W_m,i} \circ \left(\phi^{i+1}_m\right)^*
= \left(\left(\phi^i_m\right)^*\right)^{-1} 
\circ \left(\phi^i_m\right)^* \circ f^*_{Z_m,i} 
=  f^*_{Z_m,i}.
\]

Applying Lemma~\ref{lem:filter-commute} once again, we get
\[
\begin{array}{lll}
\left({\ov{p}}^i_m\right)_* \circ 
\left({\ov{q}}^i_m\right)^* \circ \left(\left(\phi^i_m\right)^*\right)^{-1} 
\circ f^*_{W_m,i} \circ \left(\phi^{i+1}_m\right)^*
& = & 
\left({\ov{p}}^i_m\right)_* \circ 
\left({\ov{q}}^i_m\right)^* \circ f^*_{Z_m,i} \\
& = & \left({\ov{p}}^i_m\right)_* \circ f^*_{Y_m,i} \circ
\left({\ov{q}}^{i+1}_m\right)^* \\
& = & f^*_{X_m,i} \circ \left({\ov{p}}^{i+1}_m\right)_* \circ 
\left({\ov{q}}^{i+1}_m\right)^*.
\end{array}
\]
This shows ~\eqref{eqn:filt-comm6} and hence the lemma.
\end{proof}
The following is an analogue of Theorem~\ref{thm:filter-Gen}
for filtrable Borel spaces.

\begin{thm}\label{thm:filter-Equiv}
Let $G$ be a linear algebraic group and let 
$\rho = \left(V_i, U_i\right)_{i \ge 1}$
be an admissible gadget for $G$ such that each ${U_i}/G$ is projective.
Let $X \in \Sch^G_k$ be a $G$-filtrable scheme with the filtration given by
~\eqref{eqn:filtration-BB}. 
Then for every $0 \le m \le n$, the inclusion $X_{m-1} \inj X_m$
induces a split exact sequence of $R(G)$-modules

\begin{equation}\label{eqn:filter-Equiv*}
0 \to \sG^{\rho}_*\left((X_{m-1})_G\right) \xrightarrow{\iota_{m-1}^*} 
\sG^{\rho}_*\left((X_{m})_G\right) \xrightarrow{j_{m}^*}
\sG^{\rho}_*\left((W_{m})_G\right) \to 0.
\end{equation}
In particular, there is an $R(G)$-module decomposition
\begin{equation}\label{eqn:filter-Equiv*1-1}
\sG^{\rho}_*(X_G) \xrightarrow{\simeq} \
\stackrel{n}{\underset{m = 0}\oplus} \ \sG^{\rho}_*\left((Z_m)_G\right).
\end{equation}
\end{thm}
\begin{proof}
It follows from ~\eqref{eqn:filt-comm4} and the left exactness of the inverse 
limit that ~\eqref{eqn:filter-Equiv*} is exact except possibly at
the right end. But Lemma~\ref{lem:filt-comm5} shows that $j^*_m \circ s^*_m$ is 
identity on $\sG^{\rho}_*\left((W_{m})_G\right)$. This yields
the split short exact sequence ~\eqref{eqn:filter-Equiv*}.

One deduces ~\eqref{eqn:filter-Equiv*1-1} from 
~\eqref{eqn:filter-Equiv*} using induction on the length of the 
filtration. For $n = 0$, the map $X_0 \to Z_0$ is a $G$-equivariant vector 
bundle and hence each map $X^i_0 \to Z^i_0$ is a vector bundle and this induces
isomorphism on the $G$-theory. We can now argue as in the proof of
Theorem~\ref{thm:filter-Gen} to complete the proof.
\end{proof}

\subsection{Filtrable schemes for torus action}\label{subsection:TCase}
Recall that a torus $T$ over $k$ is called {\sl split} if it is isomorphic
to a finite product of the multiplicative group $\G_m$ as a $k$-group
scheme. The number of copies of $\G_m$ in $T$ is called the rank of $T$.
In this subsection, $T$ will denote a split torus over $k$.
In order to apply the above structural results, we shall use
the following fundamental result, proven by Bialynicki-Birula 
\cite{BB} when $k$ is algebraically closed and by Hesselink
\cite{Hessel} when $k$ is any field. 

\begin{thm}[Bialynicki-Birula, Hesselink]\label{thm:BBH}
Let $X$ be a smooth projective scheme with an action of a split torus $T$
over $k$. Then $X$ is $T$-filtrable.
\end{thm}

\subsubsection{Good admissible gadgets}\label{subsubsection:CAG}
Let $T$ be a split torus of rank $r$. 
Given a character $\chi$ of $T$, let $L_{\chi}$ denote the one-dimensional 
representation of $T$ on which it acts via $\chi$. Given a basis
$\{\chi_1, \cdots , \chi_r\}$ of the character group $\wh{T}$ of $T$ and given
$i \ge 1$, we set $V_i = \stackrel{r}{\underset{j = 1} \prod} 
L^{\oplus i}_{\chi_j}$ and $U_i = \stackrel{r}{\underset{j = 1} \prod} 
(L^{\oplus i}_{\chi_j} \setminus \{0\})$.
Then $T$ acts on $V_i$ by $(t_1, \cdots , t_r)(x_1, \cdots , x_r)
= \left(\chi_1(t_1)(x_1), \cdots , \chi_r(t_r)(x_r)\right)$.
It is then easy to see that $\rho = \left(V_i, U_i\right)_{i \ge 1}$ is an 
admissible gadget for $T$ such that ${U_i}/T \simeq \left(\P^{i-1}\right)^r$.
The line bundle $L_{\chi_j} \stackrel{T_j}{\times} 
(L^{\oplus i}_{\chi_j}\setminus \{0\}) \to \P^{i-1}$ is the line bundle
$\sO(\pm 1)$ for each $1 \le j \le r$.
An admissible gadget for $T$ of this form will be called a 
{\sl good} admissible gadget.

Let $X \in \Sch^T_k$ be a $T$-filtrable scheme with the filtration given by 
~\eqref{eqn:filtration-BB}.
Given a good admissible gadget $\rho = \left(V_i, U_i\right)_{i \ge 1}$ for 
$T$, we see that each $X^i$ has a filtration as
in ~\eqref{eqn:filter-Equiv1} such that  
${U_i}/T = \left(\P^{i-1}_k\right)^r$ for each $i \ge 1$.
In particular, a good admissible gadget satisfies the 
hypothesis of Theorem~\ref{thm:filter-Equiv}. This implies the following.

\begin{cor}\label{cor:filter-Equiv-T}
Let $X \in \Sm^T_k$ be a $T$-filtrable scheme with the filtration given by
~\eqref{eqn:filtration-BB}. 
Then there is an $R(T)$-module decomposition
\begin{equation}\label{eqn:filter-Equiv*1}
\sK_*(X_T) \xrightarrow{\simeq} \ \stackrel{n}{\underset{m = 0}\oplus} \
\sK_*\left((Z_m)_T\right).
\end{equation}
\end{cor}

We end this section with the following result which compares the
$K$-theory and $\sK$-theory of motivic Borel spaces with torus action.
This will be generalized to the case of all connected split reductive
groups in \S~\ref{section:KBR}.

\begin{prop}\label{prop:MLTorus}
Let $X \in \Sm^T_k$ be $T$-filtrable and let $p \ge 0$.
Then for any admissible gadget 
$\rho = \left(V_i, U_i\right)_{i \ge 1}$ for $T$, the pro-abelian group
$``{{\underset{i}\varprojlim}}" \ K_p\left(X^i_T(\rho)\right)$
satisfies the Mittag-Leffler condition. In particular, the map 
\[
\tau^T_X :K_p(X_T) \to \sK_p(X_T) 
\]
is an isomorphism.
\end{prop}
\begin{proof}
In view of Theorem~\ref{thm:Stable-Ind**}, it is enough to prove the
proposition when $\rho$ is a good admissible gadget for $T$.
We shall prove in this case the stronger assertion that for any 
$T$-filtrable scheme
$X \in \Sch^T_k$, each map $f^*_{X,i}$ in the inverse system
$\left\{G_p\left(X^i_T(\rho)\right), f^*_{X,i} \right\}_{i \ge 1}$ is surjective.

Let us consider the $T$-equivariant filtration of $X$ as in 
~\eqref{eqn:filtration-BB} and the associated system of filtrations 
~\eqref{eqn:filter-Equiv1} for each $X^i = X^i_T(\rho)$. 
We prove our surjectivity 
assertion by induction on the length of the filtration.

For $0 \le m \le n$, there is a commutative diagram 
of $G$-theory of smooth schemes
\begin{equation}\label{eqn:MLTorus0}
\xymatrix@C2.9pc{
G_*\left(Z^{i+1}_m\right) \ar[r]^{f^*_{Z_m,i}} \ar[d]_{\simeq} & 
G_*\left(Z^{i}_m\right) \ar[d]^{\simeq} \\
G_*\left(W^{i+1}_m\right) \ar[r]_{f^*_{W_m,i}} & 
G_*\left(W^{i}_m\right)}
\end{equation}
where all the vertical arrows
are isomorphisms by the homotopy invariance. 

Next, we observe that $T$ acts trivially on each $Z_m$ and hence
there is an isomorphism 
$Z_m^i \simeq Z_m \times ({U_i}/T) \simeq Z_m \times \left(\P^{i-1}\right)^r$.
Hence the projective bundle formula for the $G$-theory of smooth schemes
implies that the map
$G_*\left(Z^{i+1}_m\right) \xrightarrow{f^*_{Z_m,i}} G_*\left(Z^{i}_m\right)$ is 
surjective. We conclude from ~\eqref{eqn:MLTorus0} that each $f^*_{W_m,i}$ is 
surjective.
Taking $m =0$, we see in particular that the map
$G_*(X^{i+1}_0) \to G_*(X^i_0)$ is surjective for all $i \ge 1$.

Assume now that $m \ge 1$ and that the surjectivity assertion holds for all
$j < m$. We look at the commutative diagram~\eqref{eqn:filt-comm4}.
We have shown above that the right vertical arrow in that diagram is 
surjective. The left vertical arrow is surjective by induction. We 
conclude that the middle vertical arrow in diagram~\eqref{eqn:filt-comm4}
is also surjective.
This completes the proof of the Mittag-Leffler condition.
The isomorphism of $\tau^T_X$ follows from the Mittag-Leffler
condition and the Milnor exact sequence ~\eqref{eqn:MES}.
\end{proof}

\section{$K$-theory of Borel spaces for reductive groups}
\label{section:KBR}
Our aim in this section is to prove the analogue of 
Proposition~\ref{prop:MLTorus} for any connected split reductive group.
This is achieved by reducing the problem to the case of torus using 
the push-pull operators which we now discuss.

\subsection{Push-pull operators in Equivariant $K$-theory}\label{PPO}
Let $G$ be a connected reductive group with a (not necessarily split) maximal 
torus $T$.
Let $B$ be a Borel subgroup of $G$ containing $T$.  
For any $X \in \Sch^G_k$, we know that there is
a functorial {\sl restriction} map $r^G_{T,X} : G^G_*(X) \to G^T_*(X)$ and it is
known ({\sl cf.} \cite[Theorem~1.13]{Thomason1})
that there is an isomorphism

\begin{equation}\label{eqn:Borel-T}
r^B_{T,X}: G^B_*(X) \xrightarrow{\simeq} G^T_*(X).
\end{equation}

Since $X$ is a $G$-scheme, the map
$X \stackrel{B}{\times} G \to X \times G/B$ is an isomorphism of $G$-schemes.
In particular, there is a $G$-equivariant smooth and projective map
$p_X: X \stackrel{B}{\times} G \to X$. 
The projective push-forward yields a natural {\sl induction} map
$s^G_{T,X} : G^T_*(X) \to G^G_*(X)$. 
Note that the restriction map $r^G_{T,X}$ is same as the flat pull-back
$G^G_*(X) \to G^G_*(X \stackrel{B}{\times} G) \simeq G^T(X)$.
In particular, $r^G_{T,X}$ and $s^G_{T,X}$
satisfy the usual projection formula and hence they are $K^G_*(X)$-linear.
The following result was proven by Thomason.

\begin{thm}[{\cite[Theorem~1.13]{Thomason1}}]\label{thm:PPEquiv}
For $X \in \Sch^G_k$, the composite map 
\[
G^G_*(X) \xrightarrow{ r^G_{T,X}} G^T_*(X) \xrightarrow{s^G_{T,X}}
G^G_*(X)
\]
is identity.
\end{thm}

\enlargethispage{50pt}

\subsection{Push-pull operators in $\sK$-theory}\label{PPO-S}
Our next goal is to construct the above restriction and induction maps
for the $\sK$-theory of the Borel spaces. 

\begin{lem}\label{lem:ELM-Com}
Let $f:X \to Y$ be a morphism in $\Sch^G_{{free}/k}$. Then the diagram
\begin{equation}\label{eqn:ELM-0}
\xymatrix@C.9pc{
X/B \ar[r] \ar[d] & X/G \ar[d] \\
Y/B \ar[r] & Y/G}
\end{equation}
 of quotients is Cartesian such that the horizontal maps are smooth and 
projective
in $\Sch_k$. If $f$ is a regular closed immersion (resp. of
finite Tor-dimension), then so are the  vertical maps in
~\eqref{eqn:ELM-0}.
\end{lem}
\begin{proof}
Since every principal bundle is {\'e}tale locally trivial, 
the top horizontal map is an {\'e}tale locally trivial smooth fibration with 
fiber $G/B$. Hence this map is proper by the descent property of properness.
Since this map is also quasi-projective, it must be projective. The same holds
for the bottom horizontal map. 
Proving the other properties is elementary and can be shown
using the commutative diagram
\begin{equation}\label{eqn:ELM-Com0}
\xymatrix@C1.6pc{
X \ar[r] \ar[d]_{f} & X/B \ar[r] \ar[d] & X/G \ar[d] \\
Y \ar[r] & Y/B \ar[r] & Y/G.}
\end{equation}

One easily checks that the left and the big outer squares are Cartesian
such that the vertical maps are regular closed immersions
(resp. of finite Tor-dimension). 
Since all the horizontal maps are smooth and surjective, the right square
must have the similar property. The claim about the vertical maps
in ~\eqref{eqn:ELM-0} follows from the fppf-descent property
of regular closed immersions and maps of finite Tor-dimension. 
\end{proof}

\begin{prop}\label{prop:Push-Pull**}
Let $\rho = \left(V_i, U_i\right)_{i \ge 1}$ be an admissible gadget for $G$.
Then for any $X \in \Sch^G_k$ and $p \ge 0$,
there are strict morphisms of pro-$R(G)$-modules
\[
``{{\underset{i}\varprojlim}}" \ G_p(X^i_G(\rho)) \xrightarrow{{\wt{r}}^G_{T,X}}
``{{\underset{i}\varprojlim}}" \ G_p(X^i_T(\rho)) \xrightarrow{{\wt{s}}^G_{T,X}}
``{{\underset{i}\varprojlim}}" \ G_p(X^i_G(\rho)) 
\]
such that the composite ${\wt{s}}^G_{T,X} \circ {\wt{r}}^G_{T,X}$ is identity.
Furthermore, these morphisms are contravariant for flat maps (all maps
in $\Sm^G_k$) and covariant for 
proper maps in $\Sch^G_k$. These maps satisfy the projection formula
${\wt{s}}^G_{T,X}\left(x \cdot {\wt{r}}^G_{T,X}(y)\right) = {\wt{s}}^G_{T,X}(x) 
\cdot y$ for $x \in ``{{\underset{i}\varprojlim}}" \ G_p(X^i_T(\rho))$ and 
$y \in ``{{\underset{i}\varprojlim}}" \ K_p(X^i_G(\rho))$.
\end{prop}
\begin{proof}
Since the admissible gadget $\rho$ is fixed, we drop this from our
notations in this proof. 
We first observe that for any $i \ge 1$, the map $G_*(X^i_B) \to 
G_*(X^i_T)$
is an isomorphism by the homotopy invariance. Hence, we can replace the
torus $T$ with the Borel subgroup $B$ everywhere in the proof.

\enlargethispage{50pt}

For every $i \ge 1$, there are smooth and projective maps 
$\pi^i_X:X^i_B \to X^i_G$. 
These maps induce the push-forward and pull-back maps
$(\pi^i_X)_*$ and $(\pi^i_X)^*$ on the ordinary  and
equivariant $G$-theory such that we have a commutative diagram

\begin{equation}\label{eqn:Push-Pull**0} 
\xymatrix@C2.2pc{
G_*(X^i_G) \ar[r]^{(\pi^i_X)^*} \ar[d]_{\simeq} & 
G_*(X^i_B) \ar[r]^{(\pi^i_X)_*} \ar[d]_{\simeq} &
G_*(X^i_G) \ar[d]^{\simeq} \\
G^G_*(X \times U_i) \ar[r]_{(\pi^i_X)^*} & G^B_*(X \times U_i) \ar[r]_{(\pi^i_X)_*} &
G^G_*(X \times U_i).}
\end{equation}
The functorial properties of the flat pull-back map on the $G$-theory yield
a strict morphism of the inverse systems
$\left\{G_*(X^i_G)\right\} \xrightarrow{{\wt{r}}^G_{T,X}}
\left\{G_*(X^i_B)\right\}$.

For any $i \ge 1$ and $j = i+1$, Lemma~\ref{lem:ELM-Com} yields 
a Cartesian square
\begin{equation}\label{eqn:ELM-Com1}
\xymatrix@C1.6pc{
X^i_B \ar[r]^{\pi^i_X} \ar[d]_{s^{i,j}_{B, X}} &
X^i_G \ar[d]^{s^{i,j}_{G, X}} \\
X^j_B \ar[r]_{\pi^j_X} & X^j_G}
\end{equation}
which is transverse and the horizontal maps are smooth and projective.
It follows from Lemma~\ref{lem:Tor-ind-K} that push-forward maps
$(\pi^i_X)_*$ give rise to a strict morphism of inverse systems 
$\left\{G_*(X^i_B)\right\} \xrightarrow{{\wt{s}}^G_{T,X}} 
\left\{G_*(X^i_G)\right\}$. 
The assertion that ${\wt{s}}^G_{T,X} \circ {\wt{r}}^G_{T,X}$ is identity
follows from the diagram ~\eqref{eqn:Push-Pull**0} because
the composite map on the bottom row is identity by Theorem~\ref{thm:PPEquiv}.

The covariant and contravariant functoriality of ${\wt{r}}^G_{T,X}$ and
${\wt{s}}^G_{T,X}$ follow directly from Lemmas~\ref{lem:ELM-Com} and 
~\ref{lem:Tor-ind-K}.
The projection formula for ${\wt{s}}^G_{T,X}$ and ${\wt{r}}^G_{T,X}$ follows from 
the projection formula for the maps $X^i_B \xrightarrow{\pi^i_X} X^i_G$. 
\end{proof}

As an immediate consequence of Proposition~\ref{prop:Push-Pull**}, we
obtain the following analogue of Theorem~\ref{thm:PPEquiv} for
the $\sK$-theory of the motivic Borel spaces.

\begin{cor}\label{cor:Push-Pull}
For any $X \in \Sm^G_k$ and $p \ge 0$, there are restriction and induction maps
\[
{\wt{r}}^G_{T,X}: \sK_p(X_G) \to \sK_p(X_T); 
\ \ \ \ {\wt{s}}^G_{T,X}: \sK_p(X_T) \to \sK_p(X_G)
\]
which are contravariant for all maps and covariant for proper maps 
in $\Sm^G_k$. These maps satisfy the projection formula
${\wt{s}}^G_{T,X}\left(x \cdot {\wt{r}}^G_{T,X}(y)\right) = {\wt{s}}^G_{T,X}(x)
\cdot y$ for $x \in \sK_p(X_T)$ and $y \in \sK_p(X_G)$.
Moreover, the composite map ${\wt{s}}^G_{T,X} \circ {\wt{r}}^G_{T,X}$ is 
identity.
\end{cor}

As another consequence of Proposition~\ref{prop:Push-Pull**}, we
obtain the following main result of this section.

\begin{thm}\label{thm:MLGeneral} 
Let $G$ be a connected reductive group over $k$ with a split maximal torus $T$.
Let $X \in \Sm^G_k$ be such that it is $T$-filtrable. 
Then for any admissible gadget 
$\rho = \left(V_i, U_i\right)_{i \ge 1}$ for $G$ and any $p \ge 0$,
the inverse system
$\left\{K_p\left(X^i_G(\rho)\right)\right\}_{i \ge 1}$
satisfies the Mittag-Leffler condition. In particular, the map 
\[
\tau^G_X :K_p(X_G) \to \sK_p(X_G) 
\]
is an isomorphism and hence $K_p(X_G)$ is naturally an $R(G)$-module.
\end{thm}
\begin{proof}
By \propref{prop:ML-lim-1} and ~\eqref{eqn:MES}, we only need to show the 
Mittag-Leffler condition. It follows from 
Proposition~\ref{prop:Push-Pull**} that there is a strict $R(G)$-linear
morphism $\left\{K_p(X^i_T(\rho))\right\} \xrightarrow{{\wt{s}}^G_{T,X}}
\left\{K_p(X^i_G(\rho))\right\}$ of inverse systems which is surjective.
On the other hand, the inverse system $\left\{K_p(X^i_T)\right\}$ satisfies
the Mittag-Leffler condition by Proposition~\ref{prop:MLTorus}.
Hence, this condition must hold for $\left\{K_p(X^i_G(\rho))\right\}$ as well.
\end{proof}

\section{The Atiyah-Segal completion theorem}
\label{section:Easy}
In this section, we formulate the algebraic analogue of the 
Atiyah-Segal completion theorem and prove it for filtrable schemes.
We first dispose of the baby cases of free and trivial actions. 
For a commutative ring $A$ and an $A$-module $M$, the symbol
$M[[t_1, \cdots ,t_r]]$ will denote the 
set of all formal `power series' in variables
$\{t_1, \cdots , t_r\}$ with coefficients in $M$. 
Notice that $M[[t_1, \cdots ,t_r]]$ is a module over 
the formal power series ring $A[[t_1, \cdots, t_r]]$, but it
is not necessarily same as $M {\underset{A}\otimes} \ A[[t_1, \cdots, t_r]]$.

\subsection{Case of free action}\label{subsection:Free-1}
To deal with the case of free action, we first observe the following fact.

\begin{lem}\label{lem:Nilpotent}
Let $G$ be a connected linear algebraic group over $k$. Then for any 
$X \in \Sch^G_{{free}/k}$ and $p \ge 0$, $I^m_G G^G_p(X) = 0$  for $m \gg 0$.
In particular, $G^G_p(X)$ is $I_G$-adically complete.
\end{lem}
\begin{proof}
Since $G$ is connected, it keeps each connected component of $X$ invariant
and hence it suffices to consider the case when $X$ is connected.
Setting $Y = X/G$ and letting $\mathfrak{m}_Y$ denote the ideal of
$K_0(Y)$ consisting of virtual bundles of rank zero, one has that
$I^m_G G^G_p(X) \subseteq \mathfrak{m}^m_Y G_p(Y)$ once we identify
$G^G_p(X)$ with $G_p(Y)$. Thus it suffices to show that 
$\mathfrak{m}^m_Y G_p(Y) = 0$ for $m \gg 0$.

If we consider the filtration of $G_p(Y)$ by the subgroups
\[
F_qG_p(Y) = {\underset{\dim(Z) \le q} \cup} \
{\rm Ker}\left(G_p(Y) \to G_p(Y \setminus Z)\right),
\]
then it follows from \cite[Theorem~4]{GS} that there is a pairing
\[
F^{m}_{\gamma}K_0(Y) \otimes G_p(Y) \to F_{d-m}G_p(Y),
\]
where $d = \dim(Y)$ and $F^{\bullet}_{\gamma}K_0(Y)$ is the gamma filtration on 
$K_0(Y)$. Since $\mathfrak{m}^m_Y \subseteq F^{m}_{\gamma}K_0(Y)$, we get the
desired vanishing for $m \ge d+1$.
\end{proof}

\begin{remk}\label{remk:Free**1}
Using the ideas of \cite[Proposition~4.3]{ASegal}, one can in fact show that
$G^G_*(X)$ is $I_G$-adically complete if and only if $G$ acts freely on $X$.
\end{remk}

Suppose now that $G$ is a connected linear algebraic group over $k$ and 
$X \in \Sm^G_{{free}/k}$. In this case, we identify $K^G_*(X)$ with $K_*(X/G)$.
It follows from Lemma~\ref{lem:Nilpotent} that
$K^G_*(X) \simeq \wh{{K^G_*(X)}_{I_G}}$.
Moreover, it follows from Proposition~\ref{prop:Basic-K}(2) that the map
$K_*(X/G) \to K_*(X_G)$ is an isomorphism.  
Hence, for any $p \ge 0$, we have 
\begin{equation}\label{eqn:Free0}
K^G_p(X) \xrightarrow{\simeq} \wh{{K^G_p(X)}_{I_G}}
\xrightarrow{\simeq} K_p(X_G).
\end{equation}

It follows from ~\eqref{eqn:MES} that the map $K_p(X_G) \to \sK_p(X_G)$
is surjective. To show that it is injective, we use 
Proposition~\ref{prop:Approx}.
It suffices to show using Proposition~\ref{prop:Rep-ind}
that for any admissible gadget $\rho = \left(V_i, U_i\right)$ for $G$,
the map $K_p(X/G) \to K_p\left(X^i_G(\rho)\right)$ is injective for
all large $i$. However, as $X \stackrel{G}{\times} V_i \to X/G$ 
is a vector bundle and $X^i_G(\rho) \subset X \stackrel{G}{\times} V_i$
is open, this injectivity follows from the definition of admissible gadgets
and Proposition~\ref{prop:Approx}.
We have thus shown that the maps
\begin{equation}\label{eqn:Free1}
K^G_p(X) \xrightarrow{\simeq} \wh{{K^G_p(X)}_{I_G}}
\xrightarrow{\simeq} K_p(X_G) \xrightarrow{\simeq} \sK_p(X_G)
\end{equation}
are all isomorphisms for all $p \ge 0$ if $X \in \Sm^G_{{free}/k}$.

\subsection{The case of trivial action}\label{subsection:TRT}
The case of trivial action of a split torus can be handled using some
results of Thomason. We show how this works. Let $T$ be a split torus
over $k$ of rank $r$ acting trivially on a smooth scheme $X$.
Given an admissible gadget $\rho = \left(V_i, U_i\right)$ for $T$, there is a
natural map of inverse systems 
$\left\{K^T_p(X)\right\} \to \left\{K_p(X^i_T)\right\}$
of $R(T)$-modules and this induces a map 
$K^T_p(X) \to \sK_p(X_T)$.

\begin{lem}\label{lem:Trivial-AS}
The maps $K^T_p(X) \to \sK_p(X_T) \xrightarrow{\iota^T_X} \wh{{\sK_p(X_T)}_{I_T}}$
induce isomorphisms of $\wh{R(T)}$-modules
\begin{equation}\label{eqn:Trivial-AS*1}
\wh{{K^T_p(X)}_{I_T}} \xrightarrow{\simeq} 
\sK_p(X_T) \xrightarrow{\simeq} \wh{{\sK_p(X_T)}_{I_T}} 
\xleftarrow{\simeq} K_p(X_T).
\end{equation}
\end{lem}
\begin{proof}
Since $T$ acts trivially on $X$, it follows from
\cite[Lemma~5.6]{Thomason2} that the functor 
${\rm Vec}(X) \times {\rm Rep}(T) \to {\rm Vec}^T(X)$
induces an isomorphism 
\begin{equation}\label{eqn:K-trivial}
K_*(X) {\underset{\Z}\otimes} \ R(T) \xrightarrow{\simeq} K^T_*(X).
\end{equation}

To show the first two isomorphisms of ~\eqref{eqn:Trivial-AS*1},
let $\rho = \left(V_i, U_i\right)$ be a good admissible gadget for $T$
and set $X^i_T = X^i_T(\rho)$.
For $i \ge 1$, let $J^i_T$ denote the ideal 
$\left(\rho^i_1, \cdots , \rho^i_r\right)$ of $R(T) = \Z[t_1, \cdots , t_r,
(t_1\cdots t_r)^{-1}]$, where $\rho_j = 1- t_j$.
Notice that $J^1_T = I_T$ and $J^i_T \subseteq I^i_T$ for each $i \ge 1$.
We claim for each $i \ge 1$ that there is a short exact sequence
\begin{equation}\label{eqn:Borel-complete-T2}
0 \to J^i_T\left(K^T_*(X)\right) \to K^T_*(X) \to 
K_*(X^i_T) \to 0.
\end{equation} 

Using ~\eqref{eqn:K-trivial} and the isomorphism $K_*(X^i_T) \simeq
K_*(X) {\underset{\Z}\otimes}  K_0((\P^{i-1})^r)$ 
(by the projective bundle formula), it is 
enough to show the exactness of the sequence
\begin{equation}\label{eqn:Borel-complete-T3}
0 \to J^i_T R(T) \to R(T) \to K_0((\P^{i-1})^r) \to 0.
\end{equation}
Since $R(T) = R(\G_m) \otimes \cdots \otimes R(\G_m)$,
$K_0\left((\P^{i-1})^r\right) = K_0(\P^{i-1}) \otimes \cdots 
\otimes K_0(\P^{i-1})$,
and $J^i_T = \stackrel{r}{\underset{j =1}\sum} (\rho^i_j)$, we can further
assume that $T$ has rank one.

In this case, we have the localization sequence
\begin{equation}\label{eqn:Borel-complete-T3*}
K^T_0(k) \to K^T_0(V_i) \to K_0(\P^{i-1}) \to 0
\end{equation}
and one knows by the self intersection formula ({\sl cf.}
\cite[Theorem~2.1]{VV}) that under the 
isomorphism $R(T) \xrightarrow{\simeq} K^T_0(V_i)$, the first map in 
~\eqref{eqn:Borel-complete-T3*} is multiplication by the top 
$T$-equivariant Chern class of vector bundle $V_i$.  
Since $V_i = V^i$ and since the first Chern class of $V$ in 
$R(T) = \Z[t, t^{-1}]$ is $1-t$, the Whitney sum formula shows that the 
exact sequence ~\eqref{eqn:Borel-complete-T3*} is same as 
~\eqref{eqn:Borel-complete-T3}. This proves the claim. 

The lemma follows immediately from the claim. The $R(T)$-module structure
on each $\sK_p(X_T)$ is given by taking the inverse limit of the $R(T)$-modules
$K_p(X^i_T)$.
It follows from ~\eqref{eqn:K-trivial} and ~\eqref{eqn:Borel-complete-T2}
that there is a strict isomorphism of pro-$R(T)$-modules
\begin{equation}\label{eqn:Borel-complete-T4}
``{{\underset{i}\varprojlim}}" \ \frac{K^T_p(X)}{J^i_T\left(K^T_p(X)\right)}
\xrightarrow{\simeq} 
``{{\underset{i}\varprojlim}}" \ K_p(X^i_T).
\end{equation} 

Since $J^i_T \subseteq I^i_T$ for $i \ge 1$ and since for any $i \ge 1$,
one has $I^j_T \subseteq J^i_T$ for $j \gg i$, we see that the
map of pro-$R(T)$-modules 
$``{{\underset{i}\varprojlim}}" \ \frac{K^T_p(X)}{J^i_T\left(K^T_p(X)\right)}
\to 
``{{\underset{i}\varprojlim}}" \ \frac{K^T_p(X)}{I^i_T\left(K^T_p(X)\right)}$
is an isomorphism.
We conclude that $\wh{{K^T_p(X)}_{I_T}} \xrightarrow{\simeq} \sK_p(X_T)$.
In particular, $\sK_p(X_T)$ is $I_T$-adically complete for each $p \ge 0$.

The isomorphism $K_*(X^i_T) \simeq K_*(X) {\underset{\Z}\otimes} \ 
K_0((\P^{i-1})^r)$
also shows that each map $K_p(X^{i+1}_T) \to K_p(X^i_T)$ is 
surjective.
We conclude from ~\eqref{eqn:MES} that the map
$K_p(X_T) \to \sK_p(X_T)$ is an isomorphism.

In fact, what the above shows is that $X_T \simeq X \times (\P^{\infty}_k)^r$ 
and the map
\begin{equation}\label{eqn:Borel-complete-T1}
\gamma_X: K_p(X)[[t_1, \cdots, t_r]] \to K_p(X_T); \ \ t_j \mapsto 1- \xi_j
\end{equation}
is an isomorphism of $\wh{R(T)}$-modules, where
$\xi_j \in K_0(X_T)$ is the class of the line
bundle $p^*_j\left(\sO(-1)\right)$ under the $j$th projection 
$p_j: X_T \to \P^{\infty}_k$.
\end{proof}

\begin{remk}\label{remk:K-N}
The case $X = \Spec(k)$ of Lemma~\ref{lem:Trivial-AS} was also proven 
independently by Knizel and Neshitov \cite{KN}.
\end{remk}

\vskip .3cm

\subsection{The general case of  the completion theorem}
\label{section:ASM}
Let $G$ be a linear algebraic group over $k$ and let
$X \in \Sm^G_k$. We have seen in \S~\ref{subsubsection:Kth-Borel} that 
$\sK_p(X_G)$ is an $R(G)$-module and
so, there is a natural map $\iota^G_X: \sK_p(X_G) \to \wh{{\sK_p(X_G)}_{I_G}}$,
where the latter is the $I_G$-adic completion of $\sK_p(X_G)$.
Notice that $\iota^G_X$ is contravariant functorial in $X$ and $G$.
We begin with the following.

\begin{prop}\label{prop:Borel-complete}
Let $G$ be a connected reductive group over $k$ with a split maximal torus $T$
and let $X \in \Sm^G_k$ be $T$-filtrable.
Then for any $p \ge 0$, the map
\[
\iota^G_X: \sK_p(X_G) \to \wh{{\sK_p(X_G)}_{I_G}}
\]
is an isomorphism.
\end{prop}
\begin{proof}
It follows from Corollary~\ref{cor:Push-Pull} that the composite horizontal
maps on both rows of the commutative diagram
\begin{equation}\label{eqn:Borel-complete0}
\xymatrix@C2.8pc{
\sK_p(X_G) \ar[r]^{{\wt{r}}^G_X} \ar[d]_{\iota^G_X} &
\sK_p(X_T) \ar[r]^{{\wt{s}}^G_X} \ar[d]^{\iota^T_X} & 
\sK_p(X_G) \ar[d]^{\iota^G_X} \\
\wh{{\sK_p(X_G)}_{I_G}} \ar[r]_{{\wh{r}}^G_X} & 
\wh{{\sK_p(X_T)}_{I_G}} \ar[r]_{{\wh{s}}^G_X} &
\wh{{\sK_p(X_G)}_{I_G}}}
\end{equation}
are identity. Hence, it suffices to show that the map 
$\iota^T_X$ is an isomorphism.

It follows from \cite[Corollary~6.1]{EG1} that the $I_G$-adic and the
$I_T$-adic topologies on $R(T)$ coincide. Hence for any $R(T)$-module $M$,
its $I_G$-adic and $I_T$-adic topologies coincide. Applying this to
$\sK_p(X_T)$, we see that the map $\wh{{\sK_p(X_T)}_{I_G}} \to
\wh{{\sK_p(X_T)}_{I_T}}$ is an isomorphism. 
Thus we have reduced the problem to the case of a split torus.

Let $r$ denote the rank of $T$ and let $\rho = \left(V_i, U_i\right)$ be a 
good admissible gadget for $T$ so that 
${{U_i}/T} \simeq (\P^{i-1}_k)^r$.
Let $X \in \Sm^T_k$ be a $T$-filtrable scheme with the filtration given
by ~\eqref{eqn:filtration-BB}. By Corollary~\ref{cor:filter-Equiv-T},
it suffices to prove the result when $T$ acts trivially on $X$.
But this case follows from Lemma~\ref{lem:Trivial-AS}.
\end{proof}

\subsubsection{The Atiyah-Segal map}\label{subsection:ASmap*}
Let $G$ be a linear algebraic group over $k$ and let
$\rho = \left(V_i, U_i\right)$ be an admissible gadget for $G$.
The projection map $X \times U_i 
\xrightarrow{p_i} X$ is $G$-equivariant and hence induces the pull-back
$R(G)$-linear map 
$p^*_i: K^G_*(X) \to K^G_*\left(X \times U_i\right) = K_*(X^i_G(\rho))$. 
This map is clearly compatible with the maps $f^*_{X,i}: K_*(X^{i+1}_G(\rho)) \to
K_*(X^i_G(\rho))$. This gives us a map of pro-$R(G)$-modules
$K^G_*(X) \to ``{{\underset{i}\varprojlim}}" \ K_*(X^i_G(\rho))$.
Taking the limits, we conclude that there is a natural map
\begin{equation}\label{eqn:ASmap}
\beta^G_X :  K^G_p(X) \to \sK_p(X_G)
\end{equation}
for every $p \ge 0$. This map is contravariant functorial in $X$ and $G$
and is $R(G)$-linear. In fact, $\sK_*(X_G)$ has a natural structure of
$K^G_*(X)$-module and $\beta^G_X$ is then $K^G_*(X)$-linear. 
This map will be called the {\sl Atiyah-Segal} map in tribute to
Atiyah and Segal, who studied this map in \cite{ASegal} in the 
topological context. Since $\beta^G_X$ is a morphism of $R(G)$-modules, it 
induces a natural morphism of $I_G$-adic completions:
\begin{equation}\label{eqn:ASCom}
\wh{\beta^G_X} : \wh{{K^G_p(X)}_{I_G}} \to \wh{{\sK_p(X_G)}_{I_G}}. 
\end{equation}

Suppose now that $G$ is a connected reductive group over $k$ with a split 
maximal torus $T$ and suppose that $X \in \Sm^G_k$ is $T$-filtrable.
It follows then from Proposition~\ref{prop:Borel-complete} that the map
$\wh{\beta^G_X}$ actually lifts canonically to a map of $\wh{R(G)}$-modules
$\wh{{K^G_p(X)}_{I_G}} \to \sK_p(X_G)$. We wish to prove the following.

\begin{prop}\label{prop:Alg-AS}
Let $G$ be a connected reductive group over $k$ with a split maximal
torus $T$. Let $X \in \Sm^G_k$ be $T$-filtrable.
Then for every $p \ge 0$, the morphism 
\begin{equation}\label{eqn:Alg-AS0}
\wh{\beta^G_X} : \wh{{K^G_p(X)}_{I_G}} \to \sK_p(X_G)
\end{equation}
of $\wh{R(G)}$-modules is an isomorphism.
\end{prop}
\begin{proof}
It follows from Theorem~\ref{thm:PPEquiv} and Corollary~\ref{cor:Push-Pull} 
that the composite horizontal maps on both rows of the commutative diagram 

\begin{equation}\label{eqn:Borel-complete0-ex}
\xymatrix@C2.8pc{
\wh{{K^G_p(X)}_{I_G}} \ar[r]^{{\wh{r}}^G_X} \ar[d]_{\wh{\beta^G_X}} & 
\wh{{K^T_p(X)}_{I_G}} \ar[r]^{{\wh{s}}^G_X} \ar[d]_{\wh{\beta^T_X}} &
\wh{{K_p(X_G)}_{I_G}} \ar[d]^{\wh{\beta^G_X}} \\
\sK_p(X_G) \ar[r]_{{\wt{r}}^G_X} &
\sK_p(X_T) \ar[r]_{{\wt{s}}^G_X} & 
\sK_p(X_G)}
\end{equation}
of $\wh{R(G)}$-modules are identity. Hence, it 
suffices to show that the map $\wh{\beta^T_X}$ is an isomorphism.
The arguments in the proof of Proposition~\ref{prop:Borel-complete} show
that the map $\wh{{K^T_p(X)}_{I_G}} \to \wh{{K^T_p(X)}_{I_T}}$ is an isomorphism.
This reduces the problem to the case of a split torus.
Using Theorems~\ref{thm:filter-Gen} and ~\ref{thm:filter-Equiv}, we further
reduce to the case of trivial action of a split torus. But this case
follows from Lemma~\ref{lem:Trivial-AS}.
\end{proof}

Combining Theorem~\ref{thm:MLGeneral} and 
Proposition~\ref{prop:Alg-AS}, we get the final algebraic 
Atiyah-Segal completion theorem as follows.

\begin{thm}\label{thm:MAIN}
Let $G$ be a connected and reductive group over $k$ with a split maximal
torus $T$. Let $X \in \Sm^G_k$ be $T$-filtrable.
Then for every $p \ge 0$, the morphisms 
\begin{equation}\label{eqn:Alg-AS0*}
\wh{{K^G_p(X)}_{I_G}} \xrightarrow{\wh{\beta^G_X}} 
\sK_p(X_G) \xleftarrow{\tau^G_X} K_p(X_G) 
\end{equation}
of $\wh{R(G)}$-modules are isomorphisms.
\end{thm}

Recall that a connected linear algebraic group $G$ over $k$
is called {\sl split}, if it contains a split maximal torus.
Since every smooth projective $T$-scheme is $T$-filtrable by 
Theorem~\ref{thm:BBH}, we get the following.

\begin{cor}\label{cor:MAIN-Proj}
Let $G$ be a connected and split reductive group over $k$
and let $X \in \Sm^G_k$ be projective.
Then for $p \ge 0$, there are isomorphisms of $\wh{R(G)}$-modules
\begin{equation}\label{eqn:Alg-AS0*1}
\wh{{K^G_p(X)}_{I_G}} \stackrel{\wh{\beta^G_X}}{\underset{\simeq}\to} 
\sK_p(X_G) \stackrel{\tau^G_X} {\underset{\simeq}\leftarrow} 
K_p(X_G).
\end{equation}
\end{cor}

\vskip .2cm

\begin{remk}\label{remk:Non-reductive}
One can use Proposition~\ref{prop:Unip} and Theorem~\ref{thm:Morita}
to conclude that Theorem~\ref{thm:MAIN} and Corollary~\ref{cor:MAIN-Proj} 
are true for the action of any connected
and split (not necessarily reductive) linear algebraic group in
characteristic zero.
\end{remk}

\subsection{Equivariant Quillen-Lichtenbaum conjecture}
Let $X$ be a smooth scheme over $\C$ and let $K(X, {\Z}/n)$ denote 
the mod-$n$ algebraic $K$-theory spectrum of $X$ where
$n \ge 1$. Let $K^{\rm top}_*(X, {\Z}/n)$ denote the mod-$n$
topological $K$-theory of the analytic space $X(\C)$. The
celebrated Quillen-Lichtenbaum conjecture says that
the topological realization map
\begin{equation}\label{eqn:QC}
\tau_X: K_p(X, {\Z}/n) \to K^{\rm top}_{-p}(X, {\Z}/n)
\end{equation}
is an isomorphism for $p \ge \dim(X) -1$ and a monomorphism for
$p = \dim(X) - 2$.
This conjecture is now known to be true as a consequence of the
proof of the Bloch-Kato conjecture by Rost and Voevodsky  \cite{Voevodsky}.
Combining this with our results in this section,
we get the following the solution to the equivariant version the
Quillen-Lichtenbaum conjecture.

\begin{thm}\label{thm:EQLC}
Let $G$ be a connected (not necessarily reductive) linear algebraic group over 
$\C$ containing a maximal torus $T$. 
Let $M$ be a maximal compact subgroup of the complex Lie group $G(\C)$.
Let $X$ be a smooth projective scheme with a $G$-action. Then the 
topological realization map
\begin{equation}\label{eqn:EQLC-0}
\tau^G_X: K^G_p(X, {\Z}/n) \to K^{M, \rm top}_{-p}(X, {\Z}/n)
\end{equation}
is an isomorphism for $p \ge \dim(X^T) -1$ and a monomorphism for
$p = \dim(X^T) - 2$.
\end{thm}
\begin{proof}
As remarked above, we can assume that $G$ is reductive. Using 
Theorem~\ref{thm:PPEquiv} and the analogous theorem for the
equivariant topological theory for $M$-action in \cite[Proposition~3.8]{Segal},
we can reduce to the case of torus. Next, we observe that the proof of 
Theorem~\ref{thm:filter-Gen} is motivic and works
for the mod-$n$ algebraic and topological $K$-theory, mutatis mutandis.
Using this theorem, we reduce to the case of trivial action.
The result now follows from ~\eqref{eqn:K-trivial}, 
\cite[Proposition~2.2]{Segal} and the solution of the non-equivariant 
Quillen-Lichtenbaum conjecture.
\end{proof}

\section{Completion theorem for non-projective schemes}
\label{section:NON-P}
In this section, we show that the algebraic Atiyah-Segal completion theorem
holds for $G^G_0(-)$ for all $G$-schemes and for all linear 
algebraic groups $G$. 
In order to do so, we need an analogue of Theorem~\ref{thm:Stable-Ind**}
for singular schemes. This is achieved using the following general result
of independent interest. 
The proof is a straightforward translation of the
proof of a similar result for algebraic cobordism in
\cite[Proposition~15]{HL}. We only give a brief sketch.

\begin{prop}\label{prop:Approx}
Let $X$ be a $k$-scheme and let $p : E \to X$ be a vector 
bundle of rank $r \ge 1$. Then there exists a positive integer $n$ such
that the following hold.
\begin{enumerate}
\item
The integer $n$ depends only on $X$.
\item
For any closed subscheme $Y \subsetneq E$ of codimension larger than $n$, the 
restriction map $G_i(E) \to G_i(E \setminus Y)$ is injective for all $i \ge 0$.
\end{enumerate}
\end{prop}
\begin{proof}

Set $U = E \setminus Y$ and let $j:U \to Y$ be the inclusion map.
We consider the commutative diagram
\begin{equation}\label{eqn:Approx0}
\xymatrix@C.5pc{
G_i(X) \ar[d]_{p^*} \ar[dr] & \\
G_i(E) \ar[r]_{j^*} & G_i(U)}
\end{equation}
of $G_i$-groups.
Since $p^*$ is an isomorphism by the homotopy invariance, the proposition is
equivalent to showing that the composite map $j^* \circ p^*$ is injective. 
To show this, the only new case we need to consider is when $k$ is finite and
$E$ is the trivial bundle. The rest of the proof can be easily deduced from
\cite[Proposition~15]{HL}.  

If $k$ is finite and $E$ is the trivial bundle, we can find an infinite field 
extension $k \subsetneq l$ 
which is obtained as a tower of finite extensions of $k$ of relatively prime 
degrees. Since $G$-theory commutes with direct limits, one can
easily show that if there is an element $a \in G_i(E)$ which dies in 
$G_i(U)$, then we can find finite extensions $l_1, l_2 \subsetneq l$ of 
$k$ of relatively prime degrees such that $a \in G_i(E)$ dies \nolinebreak
in $G_i(E_{l_1})$ and in $G_i(E_{l_2})$. 
The projection formula implies that
$a$ must die in $G_i(E)$. 
\end{proof}

\begin{remk}\label{remk:Approx*}
Under the hypothesis of Proposition~\ref{prop:Approx}, we can not claim that
the map $G_i(E) \to G_i(E \setminus Y)$ is an isomorphism (unless $i = 0$)
even if the codimension of $Y$ is arbitrarily large.
To see this,  just take 
$X$ to be $\Spec(k)$ and $Y$ to be the origin of an affine space $\A^r$.
Then for any $i \ge 0$, we get a short exact sequence
\[
0 \to K_i(\A^r) \to K_i(\A^r \setminus \{0\}) \to K_{i-1}(k) \to 0
\]
and we know that $K_i(k)$ is not zero in general for any $i > 0$.
\end{remk}

\begin{cor}\label{cor:Approx-Sing}
Given $X \in \Sch^G_k$ and admissible gadgets $\rho$
and $\rho'$ for the $G$-action on $X$, 
there is an isomorphism 
\[
{\underset{i}\varprojlim} \ G_0\left(X^i(\rho)\right) \
\simeq \ {\underset{i}\varprojlim} \ G_0\left(X^i(\rho')\right).
\]
\end{cor}
\begin{proof}
Since $\rho$ and $\rho'$ are the admissible gadgets for the $G$-action on $X$,
$G$ acts freely on each $X \times U_i$ and on $X \times U'_j$.
Furthermore, the map
$X \stackrel{G}{\times}(U_i \oplus V'_j) \to X \stackrel{G}{\times}U_i$
is a vector bundle and hence the map
$G_0(X \stackrel{G}{\times}U_i) \to 
G_0(X \stackrel{G}{\times}(U_i \oplus U'_j))$ is surjective.
It follows from Proposition~\ref{prop:Approx} that this map is in fact an
isomorphism for all $j \gg 0$. Taking the limit, we get

\begin{equation}\label{eqn:Morita-SK1} 
{\underset{i}\varprojlim} \ G_0(X \stackrel{G}{\times}U_i)
\xrightarrow{\simeq} {\underset{i}\varprojlim} \ {\underset{j}\varprojlim} \
G_0(X \stackrel{G}{\times}(U_i \oplus U'_j)).
\end{equation}

The same argument shows that 
\begin{equation}\label{eqn:Morita-SK2}
{\underset{i}\varprojlim} \ G_0(X \stackrel{G}{\times}U'_i)
\xrightarrow{\simeq} {\underset{i}\varprojlim} \ {\underset{j}\varprojlim} \
G_0(X \stackrel{G}{\times}(U_i \oplus U'_j)).
\end{equation}
These two isomorphisms prove the corollary.
\end{proof}

As a consequence of Corollary~\ref{cor:Approx-Sing}, we can define the
$\sG_0$-theory for schemes with group actions as follows.

\begin{defn}\label{defn:SK-Sing}
For a linear algebraic group $G$ over $k$ and $X \in \Sch^G_k$,
we define
\[
\sG_0(X_G) : =  
{\underset{i}\varprojlim} \ G_0(X \stackrel{G}{\times}U_i)
\]
where $\rho = \left(V_i, U_i\right)$ is any admissible gadget for the 
$G$-action on $X$.
\end{defn}

It follows from Lemma~\ref{lem:Tor-ind-K} that the functor 
$X \mapsto \sG_0(X_G)$ is covariant for proper maps and
contravariant for flat maps in $\Sch^G_k$.

\subsection{Case of split reductive groups}\label{subsection:Split-case}
We first consider the case of the action of split reductive groups over $k$.  
Let $T$ be a split torus of rank $r$ over $k$.

\begin{lem}\label{lem:T-complete*}
For $X \in \Sch^T_k$, the map
$G^T_0(X) \to \sG_0(X_T)$ induces an isomorphism
\[
\wh{ {G^T_0(X)}_{I_T}} \xrightarrow{\simeq} \sG_0(X_T).
\]
\end{lem}
\begin{proof}
Let $\rho = \left(V_i, U_i\right)$ be a good admissible gadget for $T$.
It suffices to show that the map 
$\wh{ {G^T_0(X)}_{I_T}} \xrightarrow{\simeq} \sG^{\rho}_0(X)$ is an isomorphism. 

For any $i \ge 0$, let $J^i(X) = 
{\rm Ker}\left(G^T_0(X) \to G_0(X^i_T(\rho))\right)$.
The surjection
\[
G^T_0(X) \xrightarrow{\simeq} G^G_0(X \times V_i) \surj G_0(X^i_T(\rho))
\] 
implies that the map $``{{\underset{i}\varprojlim}}" \ \frac{G^T_0(X)}{J^i(X)}
\to ``{{\underset{i}\varprojlim}}" \ G_0(X^i_T(\rho))$ is an isomorphism.
Hence, it suffices to show that the pro-$R(T)$-modules
$``{{\underset{i}\varprojlim}}" \ \frac{G^T_0(X)}{I^i_TG^T_0(X)}$ and
$``{{\underset{i}\varprojlim}}" \ \frac{G^T_0(X)}{J^i(X)}$ are isomorphic.
But this follows from \cite[Theorem~2.1]{EG1}.
\end{proof}

\begin{lem}\label{lem:ML-0}
Let $G$ be a connected split reductive group over $k$
and let $X \in \Sm^G_k$. Then the map
$K_0(X_G) \to \sK_0(X_G)$ is an isomorphism.
\end{lem}
\begin{proof}
It is enough to show that for an admissible gadget $\rho = 
\left(V_i, U_i\right)$ for $G$, the pro-$R(G)$-module 
$``{{\underset{i}\varprojlim}}" \ K_0(X^i_G(\rho))$
satisfies the Mittag-Leffler condition.
We first consider the case of a split torus $T$. By 
Theorem~\ref{thm:Stable-Ind**}, we can assume that $\rho$
is a good admissible gadget. In that case, we have shown in 
Lemma~\ref{lem:T-complete*} that there is an isomorphism 
$``{{\underset{i}\varprojlim}} \ \frac{"K^T_0(X)}{J^i(X)} \simeq 
``{{\underset{i}\varprojlim}}" \ K_0(X^i_T(\rho))$ of pro-$R(T)$-modules. 
Since the first pro-$R(T)$-module is Mittag-Leffler, 
so should be the second one.
The general case of connected split reductive groups now follows from
Proposition~\ref{prop:Push-Pull**}. 
\end{proof}

\begin{prop}\label{prop:Alg-AS-0}
Let $G$ be a connected split reductive group over $k$
and let $X \in \Sm^G_k$. Then the morphisms of $\wh{R(G)}$-modules
\begin{equation}\label{eqn:Alg-AS0*0}
\wh{{K^G_0(X)}_{I_G}} \xrightarrow{\wh{\beta^G_X}} 
\sK_0(X_G) \xleftarrow{\tau^G_X} K_0(X_G) 
\end{equation}
are isomorphisms.
For $X \in \Sch^G_k$, the map 
\begin{equation}\label{eqn:Alg-AS0*0-Sing}
\wh{{G^G_0(X)}_{I_G}} \xrightarrow{\wh{\beta^G_X}} \sG_0(X_G)
\end{equation}
is an isomorphism of $\wh{R(G)}$-modules.
\end{prop}
\begin{proof}
The second isomorphism of ~\eqref{eqn:Alg-AS0*0} follows from 
Lemma~\ref{lem:ML-0}. To prove the
first isomorphism, we use Theorem~\ref{thm:PPEquiv} and 
Corollary~\ref{cor:Push-Pull} and argue as in the proof of  
Theorem~\ref{thm:MAIN} to reduce to the case of a split torus.
In the singular case, it follows from Proposition~\ref{prop:Push-Pull**}
that Corollary~\ref{cor:Push-Pull} also holds for $\sG_0$ and we again
reduce to the case of split torus.
Finally, the case of split torus follows from Lemma~\ref{lem:T-complete*}. 
\end{proof}

\subsection{The general case}\label{subsection:Gen-case}
We now deduce the completion theorem for all groups $G$ and all 
$G$-schemes from the case of split reductive groups using the
Morita equivalence as follows. 

\begin{lem}\label{lem:Morita-SK}
Let $H$ be a closed normal subgroup of a linear algebraic group $G$ and let
$F = G/H$. Let $f: X \to Y$ be a morphism in $\Sch^G_k$ 
which is an $H$-torsor. Then there is an
isomorphism $\sG_0(Y_F) \simeq \sG_0(X_G)$.
\end{lem}
\begin{proof}
Let $\rho = \left(V_i, U_i\right)$ be an admissible gadget for $F$. Then
$G$ acts freely on $X\times U_i$ and the map 
$X \times U_i \to Y \times U_i$ is $G$-equivariant which is an $H$-torsor.
This in turn implies that the map $X \stackrel{G}{\times} U_i \to 
Y \stackrel{F}{\times} U_i = Y^i_F(\rho)$ is an isomorphism and so is the
induced map on $G$-theory. Thus, to prove the lemma, we only have to show that 
there is an isomorphism
\begin{equation}\label{eqn:Morita-SK0}
{\underset{i}\varprojlim} \ G_0(X \stackrel{G}{\times} U_i) 
\xrightarrow{\simeq} \sG_0(X_G).
\end{equation}

For this, we observe that if $\rho = \left(V_i, U_i\right)$ is the 
admissible gadget for $F$ as chosen above, then each $V_i$ is a $k$-rational 
representation of $G$ with $G$-invariant open subset $U_i$.
Since $H$ acts freely on $X$ and $F$ acts freely on each $U_i$, we see that 
$G$ acts freely on each $X \times U_i$. In particular,
$\rho$ is an admissible gadget for the $G$-action on $X$.
The isomorphism \eqref{eqn:Morita-SK0} now follows at once from
Corollary~\ref{cor:Approx-Sing}.
\end{proof}

\begin{cor}\label{cor:Morita-SK**}
Let $H \subseteq G$ be a closed subgroup and let $X \in \Sch^H_k$. 
Then there is an isomorphism $\sG_0(X_H) \simeq \sG_0(Y_G)$,
where $Y = X \stackrel{H}{\times} G$.
\end{cor}
\begin{proof}
The proof is exactly same as the proof of Corollary~\ref{cor:Morita1}
once we have Lemma~\ref{lem:Morita-SK}. We omit the details.
\end{proof}

\begin{thm}\label{thm:Alg-AS-All}
Let $G$ be any linear algebraic group over $k$
and let $X \in \Sm^G_k$. Then the morphisms of $\wh{R(G)}$-modules
\begin{equation}\label{eqn:Alg-AS0*-ex}
\wh{{K^G_0(X)}_{I_G}} \xrightarrow{\wh{\beta^G_X}} 
\sK_0(X_G) \xleftarrow{\tau^G_X} K_0(X_G) 
\end{equation}
are isomorphisms. For $X \in \Sch^G_k$, the map
\begin{equation}\label{eqn:Alg-AS0*-Sing}
\wh{{G^G_0(X)}_{I_G}} \xrightarrow{\wh{\beta^G_X}} 
\sG_0(X_G) 
\end{equation}
is an isomorphism of $\wh{R(G)}$-modules.
\end{thm}  
\begin{proof}
We embed $G$ as a closed subgroup of some general linear group $GL_n$ over $k$
and set $Y = X \stackrel{G}{\times} GL_n$. Then $Y \in \Sch^{GL_n}_k$ and
is smooth if $X$ is so.
Moreover, it follows from Theorem~\ref{thm:Morita} that the map
$G^{GL_n}_*(Y) \to G^G_*(X)$ is an isomorphism of $R(GL_n)$-modules.
Hence the map $\wh{{G^{GL_n}_*(Y)}_{I_{GL_n}}} \to \wh{{G^G_*(X)}_{I_{GL_n}}}$
is an isomorphism. On the other hand, we have observed before that the 
$I_{GL_n}$-adic and $I_G$-adic topologies on $G^G_*(X)$ coincide. 
We conclude that the map 
$\wh{{G^{GL_n}_*(Y)}_{I_{GL_n}}} \to \wh{{G^G_*(X)}_{I_{G}}}$
is an isomorphism. 

It follows from Corollary~\ref{cor:Morita1} that
$K_*(Y_{GL_n}) \simeq K_*(X_G)$ if $X \in \Sm^G_k$. 
It follows from Corollary~\ref{cor:Morita-SK**}
that $\sG_0(Y_{GL_n})\simeq \sG_0(X_G)$. Thus we have reduced the proof of the
theorem to the case of $GL_n$. But this case is covered by
Proposition~\ref{prop:Alg-AS-0}. 
\end{proof}

\vskip .3cm

As a consequence of Theorems~\ref{thm:ASMain} and ~\ref{thm:Alg-AS-All}, 
we get

\begin{cor}\label{cor:Alg-Top}
Given any complex linear algebraic group $G$, the realization
map
\[
K^{\rm alg}_0(B_G) \to K^{\rm top}_0(B_G)
\]
from the algebraic to the topological $K$-theory of the classifying space
$B_G$ is an isomorphism.
\end{cor}

The analogous statement for the algebraic ({\sl cf.} \cite{Krishna3})
and complex cobordism of $B_G$ is a conjecture of Yagita
({\sl cf.} \cite[Conjecture~12.1]{Yagita}).

\section{Failure of Atiyah-Segal completion theorem}
\label{section:Fail}
In this section, we show that the algebraic version of the
Atiyah-Segal completion problem has negative solution if the 
underlying smooth scheme is not filtrable. To show this, we
take our ground field to be the field of complex numbers $\C$ and $G = \G_m$.
Consider the closed subgroup $H = \{1, -1\} \subsetneq G$ and set
$X = G/H$. Then $X$ is a $G$-scheme having finite stabilizers. 

Let $\Z_2$ denote the $2$-adic completion of $\Z$.
For $m \ge 1$, let $\mu_m$ denote the subgroup of $m$-th roots of unity
in $\C^*$. Given a commutative ring $A$, an element $a \in A$ and an
$A$-module $M$, let $\ _aM$ denote the submodule of $M$ consisting of those
elements which are annihilated by $a$.
We wish to prove the following. This will produce counterexamples
to the completion theorem if we weaken the filtrability condition.

\begin{thm}\label{thm:Counter-E}
Let $X = G/H$ be the homogeneous space as above. Then 
\begin{enumerate}
\item
For $p \ge 0$, the map $K_p(X_G) \to \sK_p(X_G)$ is an isomorphism.
\item 
$\wh{{K^G_0(X)}_{I_G}} \xrightarrow{\simeq} K_0(X_G)$.
\item
For $p > 0$ odd, the map $\wh{{K^G_p(X)}_{I_G}} \to K_p(X_G)$ is 
an isomorphism.
\item
For $p > 0$ even, there is a short exact sequence
\[
0 \to \wh{{K^G_p(X)}_{I_G}} \to K_p(X_G) \to \Z_2 \to 0.
\]
\end{enumerate}
\end{thm}

We shall prove this theorem in several steps.
Let us consider the inverse system of rings 
$\left\{R_n = \frac{\Z[u]}{\left(u^n, u(u-2)\right)}\right\}_{n \ge 1}$ with the 
obvious quotient homomorphisms. 
We begin with the following elementary lemma.

\begin{lem}\label{lem:2-torsion}
For any abelian group $M$, the inverse system 
$\left\{{\rm Tor}^1_{\Z}\left(R_n, M\right)\right\}$ is isomorphic to
the inverse system $\left\{\ _{2^{n-1}}M\right\}$ whose structure maps
are given by multiplication by $2$.
\end{lem}
\begin{proof}
We first observe that $u^2 = 2u$ in $R_n$ for all $n \ge 1$.
Repeatedly applying this relation, it is easy to see that
$R_n \simeq \frac{\Z[u]}{\left(2^{n-1}u, u(u-2)\right)}$ for $n \ge 1$.
But this ring can be easily seen to be isomorphic to
$\Z \times ({\Z}/{2^{n-1}\Z})$ as an abelian group.
The lemma follows immediately from this.
\end{proof}

\begin{lem}\label{lem:2-torsion**}
Let $\left\{S_n\right\}$ denote the inverse system of rings
$\left\{S_n = \frac{\Z[t]}{\left((1-t)^n, t^2-1\right)}\right\}_{n \ge 0}$ with 
the obvious quotient homomorphisms. Then
\begin{enumerate}
\item
For $p \ge 0$ even, ${\rm Tor}^1_{\Z}\left(S_n, K_p(\C)\right) = 0$ for
each $n \ge 0$.
\item
For $p > 0$ odd, 
${\underset{n}\varprojlim} \ {\rm Tor}^1_{\Z}\left(S_n, K_p(\C)\right) \simeq
\Z_2$ and ${\underset{n}{\varprojlim}^1} \ 
{\rm Tor}^1_{\Z}\left(S_n, K_p(\C)\right) = 0$. 
\end{enumerate}
\end{lem}
\begin{proof}
There is a strict isomorphism of the inverse systems of rings
$\left\{S_n\right\} \xrightarrow{\simeq} \left\{R_n\right\}$ via the
transformation $t \mapsto 1-u$. 
The corollary is now a consequence of Lemma~\ref{lem:2-torsion} and the 
solution of the Lichtenbaum's conjecture (about the $K$-theory of algebraically 
closed fields) by Suslin \cite{Suslin}. 

Suslin has shown that for $p = 2q > 0$, $K_{p}(\C)$ is uniquely divisible.
He has also shown that for $p = 2q-1 > 0$ and $n \ge 0$, there is a 
commutative diagram
\[
\xymatrix@C2pc{
\ _{2^{n+1}}K_p(\C) \ar[r]^>>>>>{\simeq} \ar[d]_{2} & \mu_{2^{n+1}} \ar[d] \\
\ _{2^n}K_p(\C) \ar[r]_>>>>>>{\simeq} & \mu_{2^n}}
\]
such that the left vertical arrow is  multiplication by $2$, the
right vertical arrow is the obvious quotient map and the horizontal arrows
are isomorphisms ({\sl cf.} \cite[Proof of Theorem~IV.1.6]{Weibel}). 
In particular, the left vertical arrow is surjective.
The corollary now follows from this and Lemma~\ref{lem:2-torsion}.
\end{proof}

\begin{prop}\label{prop:simple}
Let $X = G/H$ be as chosen in the beginning of this section and let
$\rho = 1-t \in R(G) = \Z[t, t^{-1}]$. Then
\begin{enumerate}
\item
For $p \ge 0$ even, ${\underset{n}{\varprojlim}^m} \ _{\rho^n}K^G_p(X) = 0$
for $m = 0,1$.
\item
For $p > 0$ odd, 
${\underset{n}{\varprojlim}^1} \ _{\rho^n}K^G_p(X) = 0$ and \ \ 
${\underset{n}\varprojlim} \ _{\rho^n}K^G_p(X) \simeq \Z_2$. 
\end{enumerate}
\end{prop} 
\begin{proof}
The Morita equivalence and ~\eqref{eqn:K-trivial} imply that 
$K^G_*(X) \simeq K_*(\C) {\underset{\Z}\otimes} \  R(H)$. Moreover, we know that
$R(H) \simeq \Z[t]/{(t^2-1)}$.
We have the exact sequence 
\[
0 \to \ _{\rho^n}R(H) \to R(H) \xrightarrow{\rho^n} R(H) \to
S_n \to 0,
\]
with $S_n$ as in Lemma~\ref{lem:2-torsion**}.
This yields a commutative diagram of exact sequences
\begin{equation}\label{eqn:simple0}
\xymatrix@C.5pc{
\ _{\rho^n}R(H) {\underset{\Z}\otimes} \  K_*(\C) \ar[r] & 
R(H) {\underset{\Z}\otimes} \  K_*(\C) 
\ar[dl]_{\rho^n} \ar[r] & 
\rho^nR(H) {\underset{\Z}\otimes} \  K_*(\C) \ar[r] \ar@{=}[dl] & 0 
\\
R(H) {\underset{\Z}\otimes} \  K_*(\C) &
\rho^nR(H) {\underset{\Z}\otimes} \  K_*(\C) \ar[l] & 
{\rm Tor}^1_{\Z}\left(S_n, K_*(\C)\right)
\ar[l] \ar[u] & 0 \ar[l]}
\end{equation}
where the map ${\rm Tor}^1_{\Z}\left(S_n, K_*(\C)\right) \to 
\rho^nR(H) {\underset{\Z}\otimes} \  K_*(\C)$ is injective because the 
possible kernel of this map comes from 
${\rm Tor}^1_{\Z}\left(R(H), K_*(\C)\right)$, and 
this is zero since $R(H)$ is torsion-free.
A diagram chase gives us for $n \ge 0$, an exact sequence

\begin{equation}\label{eqn:simple1}
\ _{\rho^n}R(H) {\underset{\Z}\otimes} \  K_*(\C) \to \ _{\rho^n}K^G_*(X) \to 
{\rm Tor}^1_{\Z}\left(S_n, K_*(\C)\right) \to 0.
\end{equation}

Let $\ov{\ _{\rho^n}R(H) {\underset{\Z}\otimes} \  K_*(\C)}$ denote the image of 
the first map in ~\eqref{eqn:simple1}. We claim that the pro-abelian group
$``{{\underset{n}\varprojlim}}" \ 
\ov{\ _{\rho^n}R(H) {\underset{\Z}\otimes} \  K_*(\C)}$ is zero.
For this, it is enough to show that 
$``{{\underset{n}\varprojlim}}" \ 
\ _{\rho^n}R(H) {\underset{\Z}\otimes} \  K_*(\C) = 0$.

Since $R(H)$ is a noetherian ring, the chain 
$\ _{\rho}R(H) \subseteq \ _{\rho^2}R(H) \subseteq \cdots $ of ideals in $R(H)$
must stabilize. In other words, there exists $m \gg 0$ such that
$\ _{\rho^m}R(H) = \ _{\rho^{m+1}}R(H) = \cdots $. This implies that
$\rho^m$ annihilates $\ _{\rho^{n+m}}R(H)$ for every $n \ge 0$.
That is, the map $\ _{\rho^{n+m}}R(H) \xrightarrow{\rho^m} \ _{\rho^{n}}R(H)$
is zero. Hence, the map
\[
\ _{\rho^{n+m}}R(H) {\underset{\Z}\otimes} \  K_*(\C) \xrightarrow{\rho^m}
\ _{\rho^{n}}R(H) {\underset{\Z}\otimes} \  K_*(\C)
\]
is zero for every $n \ge 0$. This proves the claim.
We conclude from this claim and ~\eqref{eqn:simple1} that 
the map of pro-abelian groups
\[
``{{\underset{n}\varprojlim}}" \  \ _{\rho^n}K^G_*(X) \to 
``{{\underset{n}\varprojlim}}" \ {\rm Tor}^1_{\Z}\left(S_n, K_*(\C)\right)
\]
is an isomorphism.
We now apply Lemma~\ref{lem:2-torsion**} to conclude the proof.
\end{proof}

\begin{lem}\label{lem:Counter-E-step}
Let $X = G/H$ be as in Theorem~\ref{thm:Counter-E} and let 
$\rho = 1-t \in R(G) = \Z[t, t^{-1}]$. Let $\left(V_i, U_i\right)$ be a good 
admissible gadget for $G$ as in \S~\ref{subsubsection:CAG}. Then for
every $p \ge 0$ and $i \ge 1$, there is a short exact sequence of 
$R(G)$-modules
\begin{equation}\label{eqn:CE-0}
0 \to \frac{K^G_p(X)}{(\rho^i)} \to K_p(X^i_G) \to \ _{\rho^i}K^G_{p-1}(X) \to 0.
\end{equation}
\end{lem}
\begin{proof}
Using Theorem~\ref{thm:Morita} (see also \cite[Theorem~3.8, Remark~3.9]{JK}),
the lemma is equivalent to showing that there is a short exact sequence
of $R(G)$-modules
\begin{equation}\label{eqn:CE-1}
0 \to \frac{K^H_p(k)}{(\rho^i)} \to K^H_p(U_i) \to \ _{\rho^i}K^H_{p-1}(k) \to 0.
\end{equation}

In our choice of the admissible gadget $(V_i, U_i)$ for $G$ in
\S~\ref{subsubsection:CAG}, we take
$V_i = \A^i_k$ with $G$ acting by the scalar multiplication and $U_i = \A^i_k
\setminus \{0\}$.

Writing down the long exact equivariant $K$-theory localization sequence
for the inclusion $U_i \inj \A^i_k$, using the isomorphism
$K^H_*(\A^i_k) \xrightarrow{\simeq} K^H_*(k)$ via the 0-section embedding
and using the self-intersection formula \cite[Theorem~2.1]{VV},
we obtain a short exact sequence
\begin{equation}\label{eqn:CE-2}
0 \to \frac{K^H_p(k)}{(\alpha_i)} \to K^H_p(U_i) \to \ _{\alpha_i}K^H_{p-1}(k) 
\to 0,
\end{equation}
where $\alpha_i = c_i(N_{{\{0\}}/{\A^i_k}}) \in R(H)$ is the $i$-th equivariant 
Chern class of the normal bundle $N_{{\{0\}}/{\A^i_k}}$ for the inclusion 
$\{0\} \inj \A^i_k$. 

As shown in \cite[Theorem~2.1]{VV}, we have
$\alpha_i = 1 - [W_i] + [\wedge^2(W_i)] - \cdots + (-1)^i [\wedge^i(W_i)]$,
where $W_i = (V_i)^{\vee}$ for every $i \ge 1$.
Letting $H = \< \sigma\>$, we see that $V_i = V^{\oplus i}$, where $V$ is the 
1-dimensional representation of $H$ given by $\sigma(v) = -v$.
In particular, one has $[V] = [V^{\vee}]$.
Using this observation and the equalities $t = [V], \ t^2 = 1$ in $R(H)$, we get
$[\wedge^j(W_i)] = {i \choose j} t^j$ for $0 \le j \le i$.
We conclude from this that $\alpha_i = (1-t)^i = \rho^i$.
This proves the lemma.
\end{proof}

\begin{lem}\label{lem:Counter-E-step-I}
Let $X = G/H$ be as in Theorem~\ref{thm:Counter-E} and let 
$\rho = 1-t \in R(G) = \Z[t, t^{-1}]$. Let $\left(V_i, U_i\right)$ be a good 
admissible gadget for $G$ as in \S~\ref{subsubsection:CAG}. Then for
every $p \ge 0$, there is a short exact sequence of pro-$R(G)$-modules
\begin{equation}\label{eqn:CE-3}
0 \to ``{{\underset{i}\varprojlim}}" \ \frac{K^G_p(X)}{(\rho^i)} \to 
``{{\underset{i}\varprojlim}}" \ K_p(X^i_G) \to 
``{{\underset{i}\varprojlim}}" \ \ _{\rho^i}K^G_{p-1}(X) \to 0.
\end{equation}
\end{lem}
\begin{proof}
We continue to use our notations of Lemma~\ref{lem:Counter-E-step}.
The lemma is again equivalent to showing that there is a 
short exact sequence of pro-$R(G)$-modules
\begin{equation}\label{eqn:CE-4}
0 \to ``{{\underset{i}\varprojlim}}" \ \frac{K^H_p(k)}{(\rho^i)} \to 
``{{\underset{i}\varprojlim}}" \ K^H_p(U_i) \to 
``{{\underset{i}\varprojlim}}" \ \ _{\rho^i}K^H_{p-1}(k) \to 0.
\end{equation}

For any $i \ge 1$, we have the two Cartesian squares:
\begin{equation}\label{eqn:CE-5}
\xymatrix@C.5pc{
\{0\} \ar[r] \ar[d] & \A^i_k \ar[d]^{\phi_i} & & \wt{U}_{i+1}\ar[r] \ar[d] & 
\A^{i+1}_k \ar[d]^{\eta_i} \\
\A^1_k \ar[r]_{\psi_i} & \A^{i+1}_k & & U_i \ar[r] & \A^i_k,} 
\end{equation}
where $\phi_i(x_1, \cdots , x_i) = (x_1, \cdots x_i, 0)$,
$\psi_i(x) = (0, \cdots , 0, x)$, $\eta_i(x_1, \cdots, x_{i+1}) =
(x_1, \cdots , x_i)$ and $\wt{U}_{i+1} = \A^{i+1}_k \setminus {\rm Im}(\psi_i)
\inj U_{i+1}$.

These squares give us a diagram of equivariant $K$-theory spectra
\begin{equation}\label{eqn:CE-6}
\xymatrix@C.8pc{
K^H(\{0\}) \ar[r] \ar[d] & K^H(\A^{i+1}_k) \ar[r] \ar@{=}[d] &
K^H(U_{i+1}) \ar[d] \\
K^H(\A^1_k) \ar[r]^{(\psi_i)_*} \ar[d] & K^H(\A^{i+1}_k) \ar[r] \ar[d]^{\phi^*_i} &
K^H(\wt{U}_{i+1}) \ar[d] \\ 
K^H(\{0\}) \ar[r] & K^H(\A^{i}_k) \ar[r] & K^H(U_{i})}
\end{equation}
in which the three rows are localization sequences.

The top localization sequence maps to the middle one by
\cite[Remark~3.4]{Quillen}. The middle localization sequence
maps to the bottom one via the 0-sections of the vector bundle $\eta_i$.
In particular, the middle localization sequence is weak equivalent to the 
bottom sequence by the homotopy invariance.   
We should also observe that at the level of homotopy groups,
the composite left vertical arrow is
given by the multiplication by the first Chern class of the
normal bundle for the inclusion $\{0\} \inj \A^1_k$.
Arguing as in the proof of Lemma~\ref{lem:Counter-E-step}, we thus get a
commutative diagram of exact sequences
\begin{equation}\label{eqn:CE-7}
\xymatrix@C.8pc{
K^H_p(k) \ar[r]^{\rho^{i+1}} \ar[d]_{\rho} & K^H_p(k) \ar[r] \ar@{=}[d] &
K^H_p(U_{i+1}) \ar[r] \ar[d] & K^H_{p-1}(k) \ar[r]^{\rho^{i+1}} \ar[d]_{\rho} &
K^H_{p-1}(k) \ar@{=}[d] \\
K^H_p(k) \ar[r]^{\rho^{i}} & K^H_p(k) \ar[r] &
K^H_p(U_{i}) \ar[r] & K^H_{p-1}(k) \ar[r]^{\rho^{i}} &K^H_{p-1}(k).}
\end{equation}
This proves ~\eqref{eqn:CE-4} and completes the proof of the lemma. 
\end{proof}

\vskip .3cm

{\sl{Proof of Theorem~\ref{thm:Counter-E}:}}
Given $X = G/H$ as in Theorem~\ref{thm:Counter-E},
we apply Lemma~\ref{lem:Counter-E-step-I} to get a short exact sequence of 
pro-$R(G)$-modules as in ~\eqref{eqn:CE-3}. 
We also notice that the structure maps of the left pro-$R(G)$-module in this
sequence are surjective. Since $I_G = (\rho)$, 
taking the limits and using Theorem~\ref{thm:Stable-Ind**}, we conclude that
\begin{equation}\label{eqn:CE8}
0 \to \wh{{K^G_p(X)}_{I_G}} \to \sK_p(X_G) \to {\underset{i}\varprojlim} \
_{\rho^i}K^G_{p-1}(X) \to 0
\end{equation} 
is exact and ${\underset{i}{\varprojlim}^1} \ K_p(X^i_G) 
\xrightarrow{\simeq} {\underset{i}{\varprojlim}^1} \ _{\rho^i}K^G_{p-1}(X)$ 
for each $p \ge 0$. 
Combining this isomorphism with Proposition~\ref{prop:simple} and
the Milnor exact sequence ~\eqref{eqn:MES} 
\[
0 \to  {\underset{i}{\varprojlim}^1} \ K_{p+1}(X^i_G) \to K_p(X_G) \to \sK_p(X_G)
\to 0,
\]
we conclude that the map 
$K_p(X_G) \to \sK_p(X_G)$ is an isomorphism for all $p \ge 0$.
Combining ~\eqref{eqn:CE8} and Proposition~\ref{prop:simple},
we conclude that 
$\wh{{K^G_p(X)}_{I_G}} \xrightarrow{\simeq} \sK_p(X_G)$ for $p > 0$ odd,
and there is a short exact sequence
\[
0 \to \wh{{K^G_p(X)}_{I_G}} \to \sK_p(X_G) \to \Z_2 \to 0
\] 
for $p > 0$ even. This finishes the proof of Theorem~\ref{thm:Counter-E}.
$\hfill\square$

\vskip .3cm

\noindent\emph{Acknowledgments.}
The author would like to thank the anonymous referee for numerous 
encouraging comments and suggestions which were very helpful in improving the
contents of this paper.

\end{document}